\documentclass[11pt]{article}

\usepackage[utf8]{inputenc}
\usepackage[T1]{fontenc}
\usepackage{amsmath,amssymb,amsthm}
\usepackage{mathtools}
\usepackage{geometry}
\usepackage{enumitem}
\usepackage{titlesec}
\usepackage[hidelinks]{hyperref}
\hypersetup{colorlinks=true,linkcolor=blue,citecolor=blue,urlcolor=blue}

\setlist{topsep=3pt,itemsep=1pt,parsep=0pt,partopsep=0pt}
\titlespacing*{\section}{0pt}{2.0ex plus .5ex minus .2ex}{0.8ex}
\titlespacing*{\subsection}{0pt}{1.5ex plus .4ex minus .2ex}{0.6ex}
\titlespacing*{\subsubsection}{0pt}{1.2ex plus .3ex minus .2ex}{0.5ex}
\theoremstyle{plain}
\newtheorem{theorem}{Theorem}[section]
\newtheorem{proposition}[theorem]{Proposition}
\newtheorem{lemma}[theorem]{Lemma}
\newtheorem{corollary}[theorem]{Corollary}
\newtheorem*{theorem*}{Theorem}
\theoremstyle{definition}
\newtheorem{definition}[theorem]{Definition}
\newtheorem{example}[theorem]{Example}
\theoremstyle{remark}
\newtheorem{remark}[theorem]{Remark}

\newcommand{\PP}{\mathbb P}
\newcommand{\CC}{\mathbb C}
\newcommand{\RR}{\mathbb R}
\newcommand{\Sing}{\operatorname{Sing}}

\newcommand{\reg}{\mathrm{reg}}
\newcommand{\Mcal}{\mathcal M}
\newcommand{\F}{\mathcal F}
\newcommand{\Hcal}{\mathcal H}

\title{Singular Levi-flats and holomorphic foliations on projective spaces}

\author{Bruno Sc\'ardua\\
Instituto de Matem\'atica, Universidade Federal do Rio de Janeiro\\
Av. Athos da Silveira Ramos, 149, Centro de Tecnologia, Bloco C\\
Cidade Universit\'aria, Rio de Janeiro, RJ 21941-909, Brazil\\
\texttt{bruno.scardua@gmail.com}}

\date{}

\begin{document}

\maketitle

\begin{abstract}
We study codimension-one holomorphic foliations admitting invariant singular real-analytic Levi-flat varieties. Our main problem is to determine when the local meromorphic first integrals furnished by the Levi-flat geometry can be continued coherently and globalized. We introduce a pencil-singularity condition and a projective compatibility condition for local meromorphic parameters. In the projective plane these yield a characterization of foliations admitting a primitive rational first integral, while in higher-dimensional projective space they give a rational integrability criterion under a generic local-product hypothesis along the dicritical locus. We also obtain semi-global and Stein versions of these results. In a different direction, we study foliations admitting infinitely many invariant Levi-flat varieties and prove a Darboux-type alternative: either the foliation admits a rational first integral, or infinitely many Levi-flats contain a common invariant algebraic curve whose virtual holonomy is abelian. Under additional assumptions on the singularities, this leads to the existence of a closed rational, and in the absence of saddle-nodes a closed logarithmic, defining \(1\)-form.\end{abstract}

\tableofcontents


\section{Introduction}
\label{sec:introduction}

 Let $X$ be a connected complex manifold and
let
\(
f:X\longrightarrow \mathbb C
\)
be a nonconstant holomorphic function. Write
\(
f=u+iv,
\qquad
u=\operatorname{Re}(f),
\qquad
v=\operatorname{Im}(f).
\)
If $c\in\mathbb R$ is a regular value of $u$, then
\(
M=\{u=c\}
\)
is a smooth real hypersurface of $X$. Since $u$ is pluriharmonic,
\(
\partial\bar\partial u=0,
\)
the Levi form of $M$ vanishes identically. Hence $M$ is {\it Levi-flat}.

At the same time, the holomorphic function $f$ defines, on its regular
locus, a codimension-one holomorphic foliation by level complex hypersurfaces
\(
\mathcal F=\{f=\text{\rm constant}\}.
\)
The hypersurface $M$ is invariant by $\mathcal F$, and is therefore equipped with a transversely analytic real codimension one foliation by complex immersed submanifolds, more precisely  the restriction of $\mathcal F$ to $M$. There are more sophisticated examples of nonsingular real hypersurfaces invariant  by nonsingular complex codimension one holomorphic foliations,  but the interesting case appears when we allow both, $M$ and the foliation $\mathcal F$, to have singularities. Real logarithmic
foliations are an interesting class of examples (cf. Section~\ref{sec:Examples}). In this paper we look for  the converse of these two examples. More precisely, we look for sufficient (or even equivalent) conditions for a complex foliation admitting Levi-flat varieties to admit a rational first integral or 
to be given by a closed rational 1-form. For this we first recall some previous works that we shall use. 
The first we mention is due to Cerveau-Lins Neto where 
the authors prove that  a germ of codimension-one holomorphic foliation at the origin $0\in \mathbb C^n, n \geq 2$ which is tangent to a germ
of real-analytic Levi-flat hypersurface admits a nonconstant
meromorphic first integral; see \cite{CLN}. Brunella subsequently gave
a geometric proof and established an inspiring further discussion in \cite{Brunella}. 
This is the local material we need. 

Let $\F$ be a codimension one holomorphic foliation on $X$ and $p\in \Sing(\F)$ a singularity of $\F$ admitting a meromorphic first integral. Assuming that  $p$ is dicritical  choose such a minimal first integral
\(
F_p=\frac{f}{g}:U_p\dashrightarrow\PP^1
\)
and write
\(
V_\lambda=\{f-\lambda g=0\},\qquad \lambda\in\PP^1,
\)
for its local pencil.

Below we find one of the main notions we introduce in this paper.

\begin{definition}[Pencil singularity]
Let $M$ be a real-analytic hypersurface invariant by $\F$ and satisfying
$\Sing(M)\subset\Sing(\F)$. Let
\(
p\in M\cap\Sing(\F)
\)
be a dicritical singularity, and let
\(
F_p=\frac{f}{g}:U_p\dashrightarrow\PP^1
\)
be a minimal local meromorphic first integral of the germ $(\F,p)$.
We say that $p$ is a \emph{pencil singularity for $M$} if:
\begin{enumerate}[label=\textup{(\roman*)}]
\item
\[
p\in\overline L
\]
for every Levi leaf $L\subset M_{\reg}$;
\item for every generic global Levi leaf $L\subset M_{\reg}$, the
restriction of $F_p$ to $L\cap U_p$ is constant (where the
restriction is understood off the indeterminacy set of $F_p$).
\end{enumerate}
Equivalently, for every generic global Levi leaf $L$ there exists a
unique $\lambda(L)\in\PP^1$ such that every local component of
$L\cap U_p$ sufficiently near $p$ is contained in
\[
V_{\lambda(L)}=\{f-\lambda(L)g=0\}.
\]
This property is independent of the chosen minimal first integral,
since two minimal generators differ by a M\"obius transformation.
\end{definition}

With this notion we obtain  the two-dimensional characterization that motivates the
rest of the paper:

\begin{theorem}[Characterization in the projective plane]
\label{thm:plane-characterization}
Let $\F$ be a holomorphic foliation on $\PP^2$. The following
conditions are equivalent.

\begin{enumerate}[label=\textup{(\roman*)}]
\item $\F$ admits a primitive nonconstant rational first integral
\[
R:\PP^2\dashrightarrow\PP^1.
\]

\item There exists an irreducible real-analytic variety of pure real codimension one 
\(
M\subset\PP^2
\)
invariant by $\F$, with
\(
\Sing(M)\subset\Sing(\F),
\)
and all the dicritical singularities
\(
p_1,\ldots,p_r\in M\cap\Sing(\F)
\)
of $\F$ on $M$, together with local meromorphic first integrals
\[
H_j:(U_j,p_j)\dashrightarrow\PP^1,
\qquad j=1,\ldots,r,
\]
such that every Levi leaf $L\subset M_{\reg}$ accumulates at least
one of the points $p_j$, and:
\begin{enumerate}[label=\textup{(\alph*)}]
\item for every generic global Levi leaf $L\subset M_{\reg}$ approaching
$p_j$, the restriction $H_j|_{L\cap U_j}$ is constant; denote its value
by $\lambda_j(L)$;
\item there exist transformations
\[
A_{ji}\in\operatorname{PGL}(2,\CC)
\]
satisfying
\[
A_{ii}=\operatorname{id},
\qquad
A_{kj}\circ A_{ji}=A_{ki}
\]
whenever the corresponding continuations are defined;
\item whenever a generic global Levi leaf $L$ accumulates at both
$p_i$ and $p_j$, one has
\[
\lambda_j(L)=A_{ji}\bigl(\lambda_i(L)\bigr).
\]
\end{enumerate}
\end{enumerate}
\end{theorem}

A particularly simple consequence is the typical situation, when we have a single dicritical singularity.

\begin{corollary}[A single dicritical singularity]
\label{cor:single-dicritical-plane}
Let $\F$ be a codimension-one holomorphic foliation on $\PP^2$
preserving an irreducible real-analytic variety of pure real codimension one  $M$. Assume
\[
\Sing(M)\subset\Sing(\F).
\]
Suppose that $p\in M\cap\Sing(\F)$ is the only dicritical
singularity of $\F$ on $M$ and that $p$ is a pencil singularity for $M$.
Then $\F$ admits a nonconstant rational first integral on $\PP^2$.
\end{corollary}

The higher-dimensional projective criterion is the following.

\begin{theorem}[Projective integrability criterion]
\label{thm:projective-main}
Let $\F$ be a codimension-one holomorphic foliation on $\PP^n$,
$n\ge3$, and let $M\subset\PP^n$ be an irreducible real-analytic variety of pure real codimension one invariant by $\F$. Assume:

\begin{enumerate}[label=\textup{(\roman*)}]
\item
\(
\Sing(M)\subset\Sing(\F);
\)

\item all dicritical singularities of $\F$ contained in $M$ belong to
an irreducible codimension-two analytic subset
\(
D\subset M\cap\Sing(\F),
\)
and $D$ is generically of local product type;

\item there exists a pencil singularity for $M$ in the local-product
locus of $D$;
\end{enumerate}

Then $\F$ admits a nonconstant rational first integral
\[
R:\PP^n\dashrightarrow\PP^1.
\]
\end{theorem}

The local product hypothesis is discussed in
Section~\ref{sec:dicritical-component}. It includes the Kupka case and,
on the generic local-product locus of an irreducible component, fixes
the transverse analytic type; see Lins Neto \cite{LNproduct}. Notice
that Theorem~\ref{thm:projective-main} does not require the
no-transverse-degree-jump condition introduced later for the semi-global
statement: its proof uses a generic linear two-plane avoiding the
exceptional set $D\setminus D^\circ$.

A useful concrete criterion for the last hypothesis of Theorem~\ref{thm:projective-main} comes
from a real transverse first integral (cf. Definition~\ref{definition:transvseparation}).

\begin{corollary}[Local transverse separation criterion]
\label{cor:local-separation}
Let $\F$ be a codimension-one holomorphic foliation on $\PP^n$,
$n\ge3$, and let $M\subset\PP^n$ be an irreducible real-analytic variety of pure real codimension one invariant by $\F$. Assume that
\[
\Sing(M)\subset\Sing(\F),
\]
that all dicritical singularities of $\F$ contained in $M$ belong to
an irreducible codimension-two analytic subset
\[
D\subset M\cap\Sing(\F)
\]
which is generically of local product type. Assume that each Levi leaf
$L\subset M_{\reg}$ accumulates at some point in the local-product
locus of $D$. Fix a point $p$ in this locus and let
\[
F_p:U_p\dashrightarrow\PP^1
\]
be a minimal local meromorphic first integral of the germ $(\F,p)$.
Suppose that, after shrinking $U_p$, there exist a real-analytic
one-dimensional manifold $S$ and a real-analytic map
\[
h:M_{\reg}\cap U_p\longrightarrow S
\]
which separates the nearby Levi leaves at $p$ and has the following
global-leaf compatibility property: for every generic global Levi leaf
$L\subset M_{\reg}$, the restriction $h|_{L\cap U_p}$ is constant.
Then $p$ is a pencil singularity for $M$, and consequently $\F$ admits a
nonconstant rational first integral on $\PP^n$.
\end{corollary}

In Section~\ref{sec:Examples} we present examples illustrating
the hypotheses of Theorem~\ref{thm:plane-characterization},
Theorem~\ref{thm:projective-main}, and
Corollary~\ref{cor:local-separation}.

Our results are related to the projective and local studies of
Fern\'andez-P\'erez, Mol and Rosas \cite{FMRdyn,FMRsing,FMRchow}. In
particular, the Chow-type theorem of \cite{FMRchow} assumes that the
Levi leaves are contained in algebraic hypersurfaces, whereas our
hypotheses are formulated in terms of the dicritical geometry and the
compatibility of local meromorphic first integrals. The two settings
are therefore substantially different.

Finally, we obtain the following Darboux type theorem for projective foliations admitting infinitely many invariant Levi-flat varieties, 
proved in Section~\ref{sec:many-leviflats-plane}.

\begin{theorem}[Infinite-family alternative]
\label{thm:leviflat-intersection-alternative}
Let $\F$ be a holomorphic foliation on $\PP^2$ and let
$\{M_\alpha\}_{\alpha\in A}$ be infinitely many pairwise distinct
irreducible invariant real-analytic Levi-flat varieties. Assume that
there is one member $M_0$ for which
\(
(M_0\cap M_\alpha)\setminus\Sing(\F)\ne\varnothing
\)
for infinitely many $\alpha$. Assume moreover that
\(
\operatorname{Sing}(M_\alpha)\subset\operatorname{Sing}(\mathcal F)
\)
for every $\alpha\in A$. Then either $\F$ admits a nonconstant
rational first integral, or there is an irreducible invariant algebraic
curve $C$ contained in infinitely many $M_\alpha$. In the second
alternative the virtual holonomy group of every connected regular leaf
component of $C\setminus\Sing(\F)$ is abelian.
\end{theorem}
\begin{corollary}
\label{corollary:abelian}
Under (ii) in Theorem~\ref{thm:leviflat-intersection-alternative}, assume, in addition, that the singularities of $\F$ on $C$ satisfy the following conditions:
\begin{enumerate}
\item every non-dicritical singularity of $\F$ on $C$ is reduced and nondegenerate, i.e., it is of the form
\[
x\,dy-\lambda y\,dx+(\cdots)=0,
\qquad \lambda\in\mathbb C^*\setminus\mathbb Q_{>0};
\]
\item the dicritical singularities are $p_1,\ldots,p_m\in C\cap\Sing(\F)$ and, after one blow-up at each $p_j$, every singularity of the transformed foliation lying on the strict transform of $C$ is reduced, non-dicritical, and is not a saddle-node.
\end{enumerate}
 Put
\[
d_j=\operatorname{mult}_{p_j}(C).
\]
If
\[
\sum_{j=1}^m d_j<\deg C,
\]
then $\F$ is defined on $\PP^2$ by a nonzero closed rational $1$-form.
If, in addition, the reduction of singularities of $\F$ contains no
saddle-nodes, then $\F$ is defined by a closed logarithmic rational
$1$-form.
\end{corollary}
Corollary~\ref{corollary:abelian} is proved right after Theorem~\ref{thm:leviflat-intersection-alternative}; the final assertion follows from
Lemma~\ref{lem:closed-rational-logarithmic}. 

The paper is organized as follows. Section~\ref{section:definitions} contains the basic material and definitions we shall refer to in our main results. 
Section~\ref{sec:Examples}
presents the motivating pencil models, primitive first integrals, and
the transverse-separation model used in Corollary~\ref{cor:local-separation};
it also illustrates the two geometric branches of the infinite-family
results. Sections~\ref{sec:continuation}--\ref{sec:semiglobal} carry out
the continuation and semi-global extension. Sections~\ref{sec:pseudoconvexity}--\ref{sec:globalization}
prove local pseudoconvexity and projective globalization. The subsequent
sections contain the real-transverse and holonomy-rigidity results and
the projective-plane consequences for infinite families of invariant
Levi-flats. The remaining  sections are dedicated to   Stein version and multiple-dicritical singularities
variants, as well as  the proof of Theorem~\ref{thm:plane-characterization}.

\section{Levi-flat varieties and complex foliations}
\label{section:definitions}

We shall extend to the singular case the classical notion of Levi-flat hypersurface in a connected complex manifold (see \cite{BarrettForness}). 
Let  $M\subset X$ be a  real-analytic subvariety of pure real codimension one  in a complex
manifold $X$. Denote by $M_{\reg}\subset M$ its regular (nonsingular) part. The {\it complex tangent distribution} on $M_{\reg}$ is
\[
T^{\CC}M=TM\cap J(TM).
\]
 The subvariety $M$ will be  called \emph{Levi-flat} when its complex distribution is
integrable; equivalently, its Levi form vanishes identically on
$M_{\reg}$ (see \cite{BarrettForness}). The resulting foliation of $M_{\reg}$ by complex
hypersurfaces is the Levi foliation. In the projective setting this
notion is naturally tied to holomorphic foliations. In particular,
Lins Neto's results \cite{LinsNeto1999} show that {\em smooth} real-analytic
projective Levi-flats are strongly constrained; in dimensions
$n\ge3$ they do not occur. This is one reason why singular
real-analytic Levi-flat varieties  are the natural objects in the
higher-dimensional projective problem considered here.

We also extend to the singular case the  natural notion of invariance by a foliation. 
Let $\F$ be a codimension one holomorphic foliation with singularities on a connected complex manifold $X$.
A real-analytic subvariety of pure real codimension one $M\subset X$ is \emph{invariant by $\F$} if, for every
$x\in M\setminus\Sing(\F)$, the leaf $\F_x$ through $x$ is contained
in $M$. Under the additional condition
\[
\Sing(M)\subset\Sing(\F),
\]
this is equivalent, for every
$x\in M_{\reg}\setminus\Sing(\F)$, to
\(
T_x\F\subset T_xM.
\)
Since $T_x\F$ and the complex tangent hyperplane
$T_x^{\CC}M$ have the same complex dimension, one has
\(
T_x\F=T_x^{\CC}M.
\)
Thus $M_{\reg}$ is automatically Levi-flat and its Levi foliation is
the restriction of $\F$.

We shall use the following standard notion of a minimal local meromorphic first
integral. Let  $(\F,p)$ be a germ of codimension one holomorphic foliation with a singularity at $p$. If the germ $(\F,p)$ admits a nonconstant meromorphic first
integral, then there exists a meromorphic first integral
\[
F_p:(X,p)\dashrightarrow\PP^1
\]
such that every meromorphic first integral $H$ of $(\F,p)$ can be
written
\[
H=R\circ F_p,\qquad R\in\CC(z).
\]
Such an $F_p$ is called a \emph{minimal meromorphic first integral}.
The existence of a minimal generator for the field of meromorphic first
integrals is due to Cerveau--Mattei; see \cite[Chap.~V]{CM}.
Equivalently, $F_p$ is minimal precisely when it admits no nontrivial
factorization
\[
F_p=R\circ G,\qquad \deg R>1,
\]
with $G$ another meromorphic first integral. If $F_p$ and
$\widetilde F_p$ are both minimal, then
$\widetilde F_p=R\circ F_p$ and $F_p=S\circ\widetilde F_p$ for
rational functions $R,S$. Hence $S\circ R=\operatorname{id}$, so
$R$ has degree one. Thus two minimal first integrals differ by a
M\"obius transformation. The field-theoretic argument is an
application of L\"uroth's theorem; see \cite{CM}.

\section{Examples}
\label{sec:Examples}

This section presents examples and local constructions illustrating the
hypotheses of the main results. We begin with rational pencils, discuss
primitivity and the geometry of their base loci, and conclude with the
transverse-separation model used in
Corollary~\ref{cor:local-separation}.

\subsection{Rational pencils}
\label{subsec:basic-pencil}
A nonconstant rational function
$R:\PP^m\dashrightarrow\PP^1$ will be called \emph{primitive} if its
generic fiber is connected. This is the usual notion; see, for instance, \cite{GalindoMonserrat}.

\begin{example}
\label{ex:pencil}
Let
\[
R=\frac{F}{G}:\PP^n\dashrightarrow\PP^1,
\]
where $F$ and $G$ are homogeneous polynomials of the same degree, and
suppose that $R$ is a rational first integral of a codimension-one
foliation $\F$.  Let $\gamma\subset\PP^1$ be a real-analytic curve and
set
\[
M=\overline{R^{-1}(\gamma)}.
\]
Away from the critical and indeterminacy loci of $R$, use local
coordinates
\[
(z_1,\ldots,z_{n-1},t),\qquad t=R.
\]
If $\rho(t,\overline t)=0$ is a real-analytic equation for $\gamma$,
then $M$ is locally defined by the same equation and
\[
T^{1,0}M=\ker dR.
\]
Its Levi leaves are the connected components of
\[
R^{-1}(\lambda),\qquad \lambda\in\gamma.
\]

For $\gamma=\mathbb RP^1$ one has
\[
M=\overline{R^{-1}(\mathbb RP^1)},
\]
and, away from the base locus $\{F=G=0\}$,
\[
F\overline G-\overline F\,G=0.
\]
The Levi leaves are
\[
F-\lambda G=0,\qquad \lambda\in\mathbb RP^1.
\]
When $D=\{F=G=0\}$ is irreducible of codimension two and generically
Kupka (hence generically of local product type), this is the basic model for the dicritical configuration in
Theorem~\ref{thm:projective-main}.

The same construction also anticipates the two alternatives for
infinite families of invariant Levi-flats on $\PP^2$. Suppose now that
$n=2$. If two real-analytic curves $\gamma_1,\gamma_2\subset\PP^1$
meet at a regular value $\lambda$, then
\[
\overline{R^{-1}(\gamma_1)}\cap
\overline{R^{-1}(\gamma_2)}
\]
contains the algebraic fiber $R^{-1}(\lambda)$. Thus intersections of
parameter curves produce invariant algebraic leaves. On the other hand,
one may choose infinitely many pairwise disjoint real-analytic curves
in an affine coordinate of the parameter sphere, for example concentric
circles
\[
|z|=r_j,\qquad r_i\ne r_j.
\]
Their inverse images are pairwise separated away from the base locus
and meet there only through the indeterminacy of the pencil. This is the
basic model for the separated pencil-type situation considered in
Section~\ref{sec:many-leviflats-plane}.
\end{example}

\subsection{Primitive pencils and pencil singularities}
\label{subsec:primitive-pencils}

We first recall the notion of primitivity in the form used below. A
nonconstant local meromorphic germ
\[
F:(X,p)\dashrightarrow(\PP^1,\lambda_0)
\]
is called \emph{primitive} if, after choosing a local coordinate on
the target, it admits no factorization
\[
F=\varphi\circ G
\]
with $G$ meromorphic and $\varphi$ a finite holomorphic germ of degree
strictly larger than one. After resolving the indeterminacy and
restricting to suitable representatives, the analytic Stein
factorization has the form
\[
\widetilde X\xrightarrow{\,G\,}S\xrightarrow{\,\varphi\,}\PP^1,
\]
where $G$ has connected fibers and $\varphi$ is finite; see
\cite[Ch.~10, \S6]{GR}. Thus primitivity is equivalent to the finite
factor having degree one, or, equivalently, to connected generic
fibers. The same description applies to a global rational map after
resolution of its indeterminacy. We shall use ``primitive rational
first integral'' in this sense.

Consider now a meromorphic first integral of the form
\[
R=\frac{f^m}{g^n},
\qquad m,n\ge1,
\]
written locally with $f$ and $g$ having no common divisorial factor.

The following conditions give a convenient generic model:
\[
\gcd(m,n)=1,
\]
the base locus
\[
B=\{f=g=0\}
\]
is irreducible of pure codimension two, $df\wedge dg\neq0$ generically
along $B$, the pencil is primitive, and the smooth real-analytic curve
$\gamma\subset\PP^1$ defining
\[
M=\overline{R^{-1}(\gamma)}
\]
avoids the multiple-fiber values $0,\infty$ and the ordinary critical
values of $R$ away from the base locus. Such curves are plentiful; in
the generic projective situation only finitely many values have to be
avoided, and after a change of coordinate on $\PP^1$ one may take, for
instance, a circle contained in $\CC^*$.

These assumptions also explain why the standing condition
\[
\Sing(M)\subset\Sing(\F)
\]
is natural in pencil examples. Since $\gamma$ avoids the bad values,
$M$ is smooth away from $B$. On the other hand, after saturation the
foliation defined by $R=f^m/g^n$ is locally represented by
\[
\omega=m g\,df-n f\,dg.
\]
Hence $\omega$ vanishes along $B$, and therefore
\[
\Sing(M)\subset B\subset\Sing(\F).
\]
At points where $df\wedge dg\neq0$ one moreover has
\[
d\omega=-(m+n)\,df\wedge dg\neq0,
\]
so the base locus is generically Kupka and therefore generically of local product type. The relevant fibers over
$\gamma$ are connected and smooth after resolution and hence
irreducible. Consequently the closure of each corresponding Levi leaf
contains the generic base locus. Moreover, $R$ is constant on every
such global Levi leaf. Hence, after replacing $R$ locally by a minimal
generator, every generic point of $B$ is a pencil singularity for $M$.

The coprimality condition on $m$ and $n$ excludes the elementary
factorization
\[
\frac{f^m}{g^n}
=
\left(\frac{f^{m/d}}{g^{n/d}}\right)^d,
\qquad d=\gcd(m,n)>1,
\]
while primitivity excludes the remaining nontrivial compositeness of
the meromorphic map.

\subsection{Logarithmic foliations}
\label{subsec:logarithmic-leviflats}

A natural class of examples is provided by {\em logarithmic foliations}. Let
$X$ be a connected complex manifold and let
$f_j\in\mathcal O(X)$, $\lambda_j\in\mathbb C^*$, $j=1,\ldots,r$. Consider
\(
\eta=\sum_{j=1}^r\lambda_j\frac{df_j}{f_j}.
\)
This is a closed logarithmic $1$-form, holomorphic on
\(
X^\circ=X\setminus D,\) with polar divisor \(
D=\bigcup_{j=1}^r\{f_j=0\}.
\)
The associated codimension one singular holomorphic foliation on $X$ will be denoted by $\mathcal F_\eta$. If one
clears denominators in order to obtain a holomorphic defining form, common
divisorial factors are understood to be removed. We may assume that this form is given by $\Omega=(f_1\cdots f_m)\eta$. 

On every simply connected open set $U\subset X^\circ$ one may choose a
holomorphic primitive
\[
F_U=\sum_{j=1}^r\lambda_j\log f_j,
\qquad
\eta|_U=dF_U.
\]
Hence the local leaves of $\mathcal F_\eta$ are the level hypersurfaces
of $F_U$. On overlaps, two such primitives differ by an additive
constant. Analytic continuation therefore defines an additive monodromy
(or period) group. In the usual reduced logarithmic representation, in
which the divisors $\{f_j=0\}$ are the irreducible components of the
polar divisor $D$, this group is
\[
\Gamma=
\left\langle
2\pi i\lambda_1,\ldots,2\pi i\lambda_r
\right\rangle
\subset\mathbb C.
\]
More generally, one may replace $\Gamma$ below by the actual period group
of $\eta$.

We first observe that the appearance of a one-dimensional real-analytic
set in the transverse parameter is not an additional hypothesis: it is a
local consequence of invariance.

\begin{lemma}[Invariant Levi-flats and logarithmic transverse structure]
\label{lem:leviflat-transverse-set}
Let $M\subset X$ be a real-analytic variety of pure real codimension one
invariant by $\mathcal F_\eta$. At every point
\[
p\in M\setminus\bigl(D\cup\Sing(\mathcal F_\eta)\bigr)
\]
there are a neighborhood $U$ of $p$, a holomorphic primitive
\[
F_U:U\longrightarrow\mathbb C,
\qquad dF_U=\eta,
\]
and a real-analytic subset $\gamma_U\subset\mathbb C$ of pure real
dimension one such that
\[
M\cap U=F_U^{-1}(\gamma_U).
\]
Moreover, if $F_V=F_U+c_{UV}, \, c_{UV}\in \mathbb C$ on an overlap $U\cap V \ne \emptyset$, then the corresponding
transverse sets satisfy
\[
\gamma_V=\gamma_U+c_{UV}.
\]
\end{lemma}

\begin{proof}
At a point outside $D\cup\Sing(\mathcal F_\eta)$, the form $\eta$ is
holomorphic and nonvanishing. After shrinking the neighborhood, there is
a holomorphic submersion $F_U$ with $dF_U=\eta$, and its fibers are the
local leaves of $\mathcal F_\eta$.

Since $M$ is invariant, $M\cap U$ is saturated by these fibers. In a
holomorphic flow box this means precisely that $M\cap U$ is the inverse
image of a real-analytic subset of the one-dimensional transverse
parameter. Since $M$ has pure real codimension one, this transverse set
has pure real dimension one. Thus
\[
M\cap U=F_U^{-1}(\gamma_U).
\]
If $F_V=F_U+c_{UV}$ on an overlap, the two descriptions of the same set
$M$ give
\[
F_U^{-1}(\gamma_U)=F_V^{-1}(\gamma_V)
=F_U^{-1}(\gamma_V-c_{UV}),
\]
and therefore
\[
\gamma_V=\gamma_U+c_{UV}.
\]
\end{proof}

The following elementary lemma describes the possible invariant
real-analytic subsets of the additive transverse line.

\begin{lemma}[Period groups and singular real-analytic sets]
\label{lem:period-group-singular-transverse}
Let
\(
\Gamma=
\left\langle
2\pi i\lambda_1,\ldots,2\pi i\lambda_r
\right\rangle
\subset\mathbb C
\) be the additive subgroup. 
Then the following are equivalent:
\begin{enumerate}[label=\textup{(\roman*)}]
\item $\Gamma$ is not dense in $\mathbb C$;
\item there exists a nonempty proper closed real-analytic subset
\(
\gamma\subset\mathbb C
\)
of pure real dimension one satisfying
\(
\gamma+\Gamma=\gamma.
\)
\end{enumerate}
Thus no proper $\Gamma$-invariant real-analytic subset of pure real
dimension one exists if and only if
\(
\overline\Gamma=\mathbb C.
\)
\end{lemma}

\begin{proof}
Assume first that $\overline\Gamma=\mathbb C$, and let
$\gamma\subset\mathbb C$ be nonempty and closed with
$\gamma+\Gamma=\gamma$. Since translations depend continuously on the
translating parameter,
\[
\gamma+\overline\Gamma=\gamma.
\]
Hence $\gamma+a=\gamma$ for every $a\in\mathbb C$. Choosing
$z_0\in\gamma$ gives $z_0+\mathbb C\subset\gamma$, and therefore
$\gamma=\mathbb C$.

Conversely, suppose that $\Gamma$ is not dense. Its closure is a proper
closed subgroup of the real Lie group
$\mathbb C\simeq\mathbb R^2$. Up to a real-linear change of coordinates,
the possibilities are
\[
\{0\},\qquad
\mathbb Z v,\qquad
\mathbb Z v+\mathbb Z w,\qquad
\mathbb R v,\qquad
\mathbb R v+\mathbb Z w,
\]
where $v$ and $w$ are linearly independent over $\mathbb R$ whenever
both occur.

If the group is contained in a real line $\mathbb R v$, an affine line
parallel to $v$ is invariant. If
\[
\overline\Gamma=\mathbb R v+\mathbb Z w,
\qquad w\notin\mathbb R v,
\]
then
\[
\gamma=
\bigcup_{k\in\mathbb Z}\bigl(\mathbb R v+kw\bigr)
\]
is a proper invariant real-analytic set of pure real dimension one; in
suitable real-linear coordinates it is given by
\[
\sin(\pi y)=0.
\]
If $\overline\Gamma$ is a rank-two lattice, then in lattice coordinates
one may take
\[
\gamma=
\{\sin(\pi x)\sin(\pi y)=0\}.
\]
The trivial and discrete cyclic cases are immediate. Thus every proper
closed subgroup of $\mathbb C$ preserves a proper real-analytic subset
of pure real dimension one.
\end{proof}

The period group therefore gives a useful sufficient condition for the
existence of singular invariant Levi-flats compatible with the chosen
logarithmic primitive.

\begin{proposition}[Period-compatible Levi-flats]
\label{prop:period-compatible-leviflats}
Assume that there exists a nonempty proper real-analytic subset
\[
\gamma\subset\mathbb C
\]
of pure real dimension one satisfying
\[
\gamma+\Gamma=\gamma.
\]
Then the local sets
\[
M_{\gamma,U}=F_U^{-1}(\gamma)
\]
glue on $X^\circ$ to an invariant real-analytic Levi-flat variety
$M_\gamma$. Moreover, the translates $\gamma+c$, $c\in\mathbb C$,
produce infinitely many pairwise distinct invariant Levi-flat varieties
of the same type.

In particular, if $\Gamma$ is not dense in $\mathbb C$, then
$\mathcal F_\eta|_{X^\circ}$ admits infinitely many invariant
real-analytic Levi-flat varieties obtained in this way.
\end{proposition}

\begin{proof}
Let $\{U\}$ be a simply connected covering of $X^\circ$ and choose on
$U$ a primitive $F_U$ of $\eta$. On an overlap one has
\[
F_V=F_U+c_{UV},
\qquad c_{UV}\in\Gamma.
\]
Since $\gamma+\Gamma=\gamma$,
\[
F_V^{-1}(\gamma)
=F_U^{-1}(\gamma-c_{UV})
=F_U^{-1}(\gamma).
\]
Hence the local inverse images glue to a global real-analytic subset
$M_\gamma\subset X^\circ$, saturated by the leaves of
$\mathcal F_\eta$. Its regular part is therefore Levi-flat.

For every $c\in\mathbb C$, the translate $\gamma+c$ is again
$\Gamma$-invariant and gives another global invariant Levi-flat. Since
$\gamma$ is proper, its translation stabilizer
\[
\operatorname{Stab}(\gamma)
=\{c\in\mathbb C;\ \gamma+c=\gamma\}
\]
is a proper closed subgroup of $\mathbb C$. Hence the quotient
$\mathbb C/\operatorname{Stab}(\gamma)$ is infinite, and the family of
translates contains infinitely many distinct members.
The last assertion follows from
Lemma~\ref{lem:period-group-singular-transverse}.
\end{proof}

\begin{remark}
The converse to Proposition~\ref{prop:period-compatible-leviflats} is
false for arbitrary invariant Levi-flats. Let $R$ be a
nonconstant meromorphic first integral and choose distinct constants
$a_1,\ldots,a_r$. Then
\[
\eta=
\sum_{j=1}^r
\lambda_j\frac{d(R-a_j)}{R-a_j}
=
\left(
\sum_{j=1}^r\frac{\lambda_j}{R-a_j}
\right)dR
\]
defines, away from its polar and zero divisors, the same foliation as
$dR=0$. The coefficients $\lambda_j$ may be chosen so that the period
group of $\eta$ is dense in $\mathbb C$, while the foliation still
admits infinitely many invariant Levi-flat varieties obtained as inverse
images under $R$ of real-analytic curves in $\mathbb P^1$. 
\end{remark}

There is an important and particularly simple subcase. Suppose that the
residues are collinear over $\mathbb R$. Equivalently, there exists
$a\in\mathbb C^*$ such that
\[
a\lambda_j\in\mathbb R,
\qquad j=1,\ldots,r.
\]
Replacing $\eta$ by $a\eta$, we may assume that all residues are real.
Then
\[
h=\operatorname{Re}(F)
 =\sum_{j=1}^r\lambda_j\log|f_j|
\]
is single-valued on $X^\circ$, because all periods of $F$ are purely
imaginary, and
\[
2\partial h=\eta.
\]
Consequently every regular level
\[
M_c=\{h=c\}
\]
is an invariant Levi-flat hypersurface, and varying $c$ gives an
infinite one-parameter family.

The converse is best formulated at the level of this real-part model.

\begin{proposition}[Periods and real-part levels]
\label{proposition:logarithmic}
Let $\Pi_\eta\subset\CC$ be the actual period group of $\eta$ on
$X^\circ$. For the logarithmic foliation $\mathcal F_\eta$, the
following are equivalent:
\begin{enumerate}[label=\textup{(\roman*)}]
\item there exists $a\in\mathbb C^*$ such that
\[
\operatorname{Re}(a\tau)=0
\qquad\text{for every }\tau\in\Pi_\eta;
\]
\item after multiplication of $\eta$ by a nonzero constant, the
multivalued primitive has a single-valued real part on $X^\circ$;
\item $\mathcal F_\eta|_{X^\circ}$ admits an infinite one-parameter
family of invariant Levi-flat level hypersurfaces of the form
\[
\{\operatorname{Re}(aF)=c\},
\qquad c\in\mathbb R.
\]
\end{enumerate}
In particular, the collinearity condition
$a\lambda_j\in\RR$ for all $j$ implies these equivalent statements.
When the actual periods are generated by meridians about distinct
reduced polar components with respective residues $\lambda_j$, the
collinearity condition is also necessary.
\end{proposition}

\begin{proof}
The equivalence of \textup{(i)} and \textup{(ii)} follows from the
actual period group. If all periods become purely imaginary after
multiplication by $a$, then $\operatorname{Re}(aF)$ is single-valued and
its regular levels give \textup{(iii)}. Conversely, the expression in
\textup{(iii)} can be globally single-valued only if every period of
$aF$ has zero real part. If $a\lambda_j$ is real for every $j$,
then every period of $a\eta$ is purely imaginary, since the winding
numbers of the functions $f_j$ are integers. Under the final
meridian-generation assumption, the periods $2\pi i\lambda_j$ occur
among the generators, so the same condition is also necessary.
\end{proof}

Notice that Proposition~\ref{prop:period-compatible-leviflats} is more
general in the singular category: a non-dense period group need not be
contained in a real line. For example, a rank-two lattice preserves
singular periodic real-analytic sets, and therefore produces singular
invariant Levi-flats on $X^\circ$, although the residues are not
collinear over $\mathbb R$.

\subsection{A transverse separation model}
\label{subsec:transverse-separation}
Next we introduce and illustrate the notions in Corollary~\ref{cor:local-separation}.
\begin{definition}[Separation of nearby Levi leaves]
\label{definition:transvseparation}
Let $p$ be a dicritical singularity of $\F$ contained in an invariant irreducible real-analytic variety of pure real codimension one $M$ with
$\Sing(M)\subset\Sing(\F)$. Let $U_p$ be a neighborhood of $p$,
let $S$ be a real-analytic one-dimensional manifold, and let
\[
h:M_{\reg}\cap U_p\longrightarrow S
\]
be a real-analytic map on the regular part near $p$. We say that $h$
\emph{separates the nearby Levi leaves at $p$} if $h$ is constant on
each local Levi leaf sufficiently near $p$ and the values assigned to
distinct such leaves are distinct. Equivalently, the induced map on
the set of nearby local Levi leaves is injective. The target $S$ may,
in particular, be $\RR$ or $\mathbb{RP}^1$.
\end{definition}

\begin{example}
Consider in $\CC^2$ the local meromorphic first integral
\[
F(x,y)=\frac{x}{y}
\]
and the irreducible real-analytic variety of pure real codimension one 
\[
M=\overline{\left\{(x,y)\in\CC^2;\ y\neq0,\quad
\frac{x}{y}\in\RR\right\}}
=\{x\overline y-\overline x y=0\}.
\]
Its Levi leaves are
\[
L_\lambda=\{x-\lambda y=0\}\cap M_{\reg}
\quad(\lambda\in\RR),\qquad
L_\infty=\{y=0\}\cap M_{\reg},
\]
and the origin is a pencil singularity for $M$. On all of
$M_{\reg}$ near the origin the real-analytic map
\[
h(x,y)=[x:y]\in\mathbb RP^1
\]
is constant on the local pieces of the Levi leaves and satisfies
\[
h|_{L_\lambda}=\lambda.
\]
In the chart $y\neq0$ its real coordinate is
$\operatorname{Re}(x/y)$; the chart $x\neq0$ covers the parameter
$\lambda=\infty$. Thus $h$ separates all nearby Levi leaves. Notice that
$h$ does not extend continuously across the dicritical center $(0,0)$,
as is forced by the pencil-singularity property.
\end{example}

\section{Continuation along the regular Levi foliation}
\label{sec:continuation}

This section constructs the first integral on a neighborhood of the regular
part of the invariant hypersurface. The pencil-singularity condition is
used to show that holonomy continuation of the local meromorphic parameter
is single-valued.

\begin{proposition}
\label{prop:nearMreg}
Let $X$ be a connected complex manifold, let $\F$ be a codimension-one
holomorphic foliation on $X$, and let $M\subset X$ be a connected irreducible real-analytic variety of pure real codimension one invariant by $\F$. Let
\[
p\in M\cap\Sing(\F)
\]
be a pencil singularity for $M$, and let
\[
F_p=\frac{f}{g}:U_p\dashrightarrow\PP^1
\]
be a minimal local meromorphic first integral for the germ $(\F,p)$.
Then $F_p$ extends by holonomy to a single-valued meromorphic first
integral on an open neighborhood of $M_{\reg}$.
\end{proposition}

\begin{proof}[Proof of Proposition~\ref{prop:nearMreg}]
Fix $x\in M_{\reg}$ and denote by $L_x$ the Levi leaf through $x$.
By the accumulation hypothesis,
\[
p\in\overline{L_x}.
\]
Hence $L_x$ meets $U_p$.  Choose a regular point
$y\in L_x\cap U_p$ and a path $\gamma\subset L_x$ joining $x$ to
$y$.  Choose small complex transversals $\Sigma_x$ and $\Sigma_y$.
Holonomy along $\gamma$ defines
\[
h_\gamma:\Sigma_x\longrightarrow\Sigma_y.
\]
The meromorphic function
\[
g_\gamma=F_p\circ h_\gamma
\]
on $\Sigma_x$ extends constantly along plaques to a meromorphic first
integral $G_\gamma$ on a flow box around $x$.

Let $\gamma_1,\gamma_2$ be two such choices, and let $G_1,G_2$ be the
corresponding first integrals.  Put
\[
I=\Sigma_x\cap M_{\reg}.
\]
For generic $z\in I$, the two continuations follow the same global
Levi leaf.  By the constancy hypothesis,
\[
G_1(z)=G_2(z)
\]
for generic $z\in I$.  Since $G_1|_{\Sigma_x}$ and
$G_2|_{\Sigma_x}$ are one-variable meromorphic functions and agree on
a real set with an accumulation point,
\[
G_1|_{\Sigma_x}\equiv G_2|_{\Sigma_x}.
\]
Both are constant along plaques, hence $G_1\equiv G_2$ in the flow
box.  The same argument on overlaps gives a meromorphic function
\[
G\in\Mcal(N_{\reg})
\]
on an open neighborhood $N_{\reg}$ of $M_{\reg}$, and
\[
dG\wedge\omega=0
\]
for every local defining integrable $1$-form $\omega$ of $\F$.
\end{proof}

\section{The irreducible dicritical component}
\label{sec:dicritical-component}

We now extend the meromorphic first integral through the irreducible
dicritical component. The local-product structure provides a fixed
transverse analytic type on its generic locus, while the
pencil-singularity condition eliminates the resulting projective
monodromy.

\begin{definition}[Generic local product type]
\label{def:generic-ltp}
Let $D$ be an irreducible component of $\Sing(\F)$. Following the local
transversely product viewpoint of Lins Neto \cite{LNproduct}, we say that
$D$ is \emph{generically of local product type} if there exists a proper
complex-analytic subset
\[
E_D\subsetneq D,
\qquad \Sing(D)\subset E_D,
\]
such that
\[
D^\circ:=D\setminus E_D
\]
is connected and dense, and there exists a fixed germ $\mathcal G$ of a
singular foliation in a transverse two-dimensional space such that, near
every point of $D^\circ$, the foliation $\F$ is analytically the pull-back
of $\mathcal G$ by a submersion. We call $D^\circ$ the generic
local-product locus and $E_D$ the exceptional locus.
\end{definition}

Thus the transverse analytic type is constant along $D^\circ$, while the
exceptional locus is analytic and has positive codimension in $D$. This
formulation is compatible with the classical Kupka situation. Indeed, on
a codimension-two component the Kupka locus carries the local product
structure, while the non-Kupka locus is analytic: locally, if $\omega$
defines the foliation, it is cut out on the singular component by
$d\omega=0$. See \cite{CLNKupka,LNproduct,CAS}.

Thus one may choose minimal local meromorphic first integrals
\[
H_q:U_q\dashrightarrow\PP^1,\qquad q\in D^\circ,
\]
whose changes of transverse parameter are elements of
$\operatorname{PGL}(2,\CC)$.

\begin{lemma}
\label{lem:product-propagation}
Let $L\subset M_{\reg}$ be a Levi leaf. If
\[
\overline L\cap D^\circ\neq\varnothing,
\]
then
\[
D^\circ\subset\overline L.
\]
\end{lemma}

\begin{proof}
Set $S_L=\overline L\cap D^\circ$. It is closed in $D^\circ$. Let
$q\in S_L$. In local-product coordinates near $q$, the foliation is the
product of its fixed transverse germ with the regular foliation along
$D^\circ$. Since $L$ accumulates at $q$, the closure of the corresponding
local leaf contains the whole local plaque of $D^\circ$ through $q$.
Hence $S_L$ is open in $D^\circ$. Since $D^\circ$ is connected,
$S_L=D^\circ$.
\end{proof}

\begin{lemma}
\label{lem:monodromy}
Let $D$ be an irreducible dicritical component which is generically of
local product type, let $D^\circ$ be its connected generic local-product
locus, and fix $p\in D^\circ$. Suppose that a minimal local meromorphic first integral
\[
F_p:U_p\dashrightarrow\PP^1
\]
has the property that, for every generic global Levi leaf $L$
represented near $p$, the restriction $F_p|_{L\cap U_p}$ is constant. Then the projective
monodromy of the transverse $\PP^1$-parameter along $D^\circ$ is trivial.
\end{lemma}

\begin{proof}[Proof of Lemma~\ref{lem:monodromy}]
Normalize the parameter at $p$ by $H_p=F_p$.  Analytic continuation
around a loop $\gamma\in\pi_1(D^\circ,p)$ gives an element
\[
\rho(\gamma)\in\operatorname{PGL}(2,\CC).
\]
For a generic Levi leaf $L$ represented near $p$, let
\[
\lambda(L)=H_p(L).
\]
After continuation around $\gamma$, the same leaf is assigned the
parameter $\rho(\gamma)(\lambda(L))$.  The constancy hypothesis at $p$ gives
\[
\rho(\gamma)(\lambda(L))=\lambda(L)
\]
for every generic Levi leaf.  The local Levi-flat supplies infinitely
many such parameters with an accumulation point.  A nonidentity
M\"obius transformation has at most two fixed points.  Therefore
$\rho(\gamma)=\operatorname{id}$ for every $\gamma$, as claimed.
\end{proof}

The local first integrals $H_q$, $q\in D^\circ$, can consequently be
normalized so that
\[
\widehat H_q=A_q\circ H_q,\qquad
A_q\in\operatorname{PGL}(2,\CC),
\]
and the normalized parameter agrees with $H_p$ on generic global Levi
leaves.

\begin{definition}[No transverse degree jump]
\label{def:no-degree-jump}
We say that the generic transverse parameter along $D^\circ$ has
\emph{no transverse degree jump along $D$} if, for every
$q\in D\setminus D^\circ$ and every minimal local meromorphic first
integral
\[
H_q:U_q\dashrightarrow\PP^1
\]
of the germ $(\F,q)$, the restriction of $H_q$ near a generic point
$r\in D^\circ\cap U_q$ is still minimal. Equivalently, if
$\widehat H_r$ is a normalized minimal transverse parameter at such an
$r$ and
\[
H_q=R_q\circ\widehat H_r,
\]
then $\deg R_q=1$.
\end{definition}

\begin{proposition}
\label{prop:dicritical-extension}
The meromorphic first integral $G$ constructed in
Proposition~\ref{prop:nearMreg} extends meromorphically through every
point of $D^\circ$. If, in addition, the no-transverse-degree-jump
condition of Definition~\ref{def:no-degree-jump} holds, then $G$
extends meromorphically through every point of $D$.
\end{proposition}

\begin{proof}[Proof of Proposition~\ref{prop:dicritical-extension}]
Let $q\in D^\circ$. On the nonempty overlap of $U_q$ with
$N_{\reg}$ the functions $G$ and $\widehat H_q$ are first integrals
of the same foliation and have the same values on generic Levi leaves.
Hence they agree. Since $\widehat H_q$ is meromorphic on $U_q$, it
extends $G$ through $q$.

Now let $q\in D\setminus D^\circ$ and assume the no-transverse-degree-jump
condition. By the local theorem of Cerveau--Lins Neto there is a minimal
local meromorphic first integral
\[
H_q:U_q\dashrightarrow\PP^1.
\]
Choose a generic point $r\in D^\circ\cap U_q$. By hypothesis the
restriction of $H_q$ near $r$ is minimal. Since $\widehat H_r$ is also
minimal, the two first integrals differ by a M\"obius transformation:
\[
\widehat H_r=A_q\circ H_q,
\qquad A_q\in\operatorname{PGL}(2,\CC).
\]
The function $A_q\circ H_q$ is meromorphic on $U_q$ and agrees with
$G$ on a nonempty open set. Uniqueness of meromorphic continuation
therefore gives the extension through $q$.
\end{proof}

\begin{example}[A genuine transverse degree jump]
\label{ex:degree-jump}
The condition in Definition~\ref{def:no-degree-jump} is not automatic.
Let $d\ge2$ and consider in $(\CC^3,0)$ the meromorphic first integral
\[
H(x,y,z)=\frac{x^d}{y^d z}.
\]
The associated foliation is defined, after clearing denominators, by
\[
\omega=d yz\,dx-d xz\,dy-xy\,dz,
\]
and its singular set contains
\[
D=\{x=y=0\}.
\]
At $q=(0,0,0)$ the first integral $H$ is primitive:  For a generic $c\in\CC^*$, the fiber germ
\[
x^d-cy^dz=0
\]
is irreducible. Indeed, viewed as a polynomial in $x$ over
$\CC\{y,z\}$, it satisfies Eisenstein's criterion with respect to
the prime element $z$ \cite[\S 9.4]{DummitFoote}. On the other hand, at a nearby point
$p=(0,0,z_0)\in D$, $z_0\ne0$, choose a holomorphic branch
$a(z)$ with $a(z)^d=z$. Then
\[
G_p(x,y,z)=\frac{x}{a(z)y}
\]
is a primitive local meromorphic first integral and
\[
H=G_p^d.
\]
Thus a primitive first integral at the exceptional point $q$ restricts
near generic points of $D$ as a nonprimitive degree-$d$ quotient.
\end{example}

\section{Extension through the remaining singularities}
\label{sec:remaining-singularities}

After treating the dicritical component, it remains to cross the other
singular points of the foliation lying on $M$. These points are
non-dicritical, and their local holomorphic first integrals provide the
required extension.

Let
\[
q\in M\cap\Sing(\F),\qquad q\notin D.
\]
Such a singularity is non-dicritical.

\begin{lemma}
\label{lem:nondicritical}
The meromorphic first integral $G$ extends meromorphically to a
neighborhood of $q$.
\end{lemma}

\begin{proof}[Proof of Lemma~\ref{lem:nondicritical}]
Since $q$ is non-dicritical, a minimal local meromorphic first integral
has no indeterminacy at $q$: otherwise a reduced representation $f/g$
with $f(q)=g(q)=0$ would produce infinitely many distinct local
separatrices through $q$. After shrinking $U_q$ and choosing the
appropriate target chart, we may therefore take a minimal holomorphic
first integral
\[
H_q:U_q\longrightarrow\CC,\qquad H_q(q)=c.
\]
On the regular locus of $H_q$, invariance of $M$ means that $M$ is
saturated by the level hypersurfaces of $H_q$. Hence the image of the
regular part of the relevant local component of $M$ in the one-dimensional
leaf space is a real-analytic curve germ. After selecting that component,
there is a germ of a real-analytic curve $\gamma\subset\CC$, with
$c\in\gamma$, such that
\[
M\cap U_q=\overline{H_q^{-1}(\gamma)}.
\]
In particular,
\[
H_q^{-1}(c)\subset M.
\]

The fiber $H_q^{-1}(c)$ is a complex hypersurface through $q$. Since
$\Sing(\F)$ has complex codimension at least two, it is not contained
in $\Sing(\F)$. We may therefore choose
\[
x\in H_q^{-1}(c)\cap M,
\qquad x\notin\Sing(\F),
\]
arbitrarily close to $q$. Since $\Sing(M)\subset\Sing(\F)$, we have
$x\in M_{\reg}$.

Near $x$, both $G$ and $H_q$ are first integrals of the same regular
codimension-one foliation. Since $H_q(x)=c$, there exists a one-variable
meromorphic germ $\varphi$ at $c$ such that
\[
G=\varphi\circ H_q
\]
on a neighborhood of $x$. Therefore $\varphi\circ H_q$ is meromorphic
near $q$ and agrees with $G$ on a nonempty open subset of their common
domain. By uniqueness of meromorphic continuation, it extends $G$
through $q$.
\end{proof}

\section{Semi-global integrability near the Levi-flat}
\label{sec:semiglobal}

The preceding continuation results are assembled here into a semi-global
integrability statement. They yield a single meromorphic first integral
on a neighborhood of the entire invariant hypersurface.

\begin{theorem}[Semi-global integrability]
\label{thm:semiglobal}
Let $X$ be a connected complex manifold of dimension $n\ge3$, let $\F$ be a
codimension-one holomorphic foliation on $X$, and let $M\subset X$ be
a connected irreducible real-analytic variety of pure real codimension one  invariant
by $\F$. Assume:

\begin{enumerate}[label=\textup{(\roman*)}]
\item
\[
\Sing(M)\subset\Sing(\F);
\]

\item all dicritical singularities of $\F$ contained in $M$ belong to
an irreducible codimension-two complex analytic subset
\[
D\subset M\cap\Sing(\F),
\]
and $D$ is generically of local product type;

\item there exists a pencil singularity for $M$ in the local-product
locus of $D$;

\item the generic transverse parameter along $D^\circ$ has no
transverse degree jump along $D$ in the sense of
Definition~\ref{def:no-degree-jump}.
\end{enumerate}

Then there exist an open neighborhood $N\supset M$ and a nonconstant
meromorphic first integral
\[
G\in\Mcal(N)
\]
for $\F|_N$.
\end{theorem}

\begin{proof}[Proof of Theorem~\ref{thm:semiglobal}]
Proposition~\ref{prop:nearMreg} gives a meromorphic first integral near
$M_{\reg}$. By hypothesis 	extup{(iv)},
Proposition~\ref{prop:dicritical-extension} extends it through the
whole dicritical component $D$, while Lemma~\ref{lem:nondicritical}
extends it through the remaining points
of $M\cap\Sing(\F)$.  Since
\[
\Sing(M)\subset\Sing(\F),
\]
these extensions cover all points of $M$.  By uniqueness of
meromorphic continuation they glue to a nonconstant meromorphic first
integral
\[
G\in\Mcal(N)
\]
on some open neighborhood $N\supset M$.
\end{proof}

\section{Local pseudoconvexity of the complementary sides}
\label{sec:pseudoconvexity}

This section studies the local geometry of the components of the
complement of a Levi-flat variety $M\subset X$, including at singular
points. We use complex hypersurface germs contained in the Levi-flat;
no description of $M$ as the inverse image of a curve under a local
first integral is needed.

\begin{proposition}
\label{prop:pseudoconvex}
Let $X$ be a connected complex manifold and let $M\subset X$ be a singular Levi-flat variety. Assume that, near every point $q\in M$, the Levi foliation is defined by a local meromorphic first integral. Then every local side of $M$ is pseudoconvex, including at singular points of $M$.
\end{proposition}

\begin{proof}[Proof of Proposition~\ref{prop:pseudoconvex}]
Since $M$ has pure real codimension one, its smooth hypersurface locus
is dense in $M$. By Lebl's complex-leaf lemma
\cite[Lemma~5.4]{LeblSingularSet}, through every $q\in M$ there is a
germ of a complex hypersurface contained in $M$. Let $h_q$ be a
holomorphic equation for this hypersurface on a small coordinate
neighborhood. The germ $1/h_q$ is holomorphic on each local component
of the complement of $M$ and cannot extend holomorphically across
$q$. The local characterization of pseudoconvexity by domains of
holomorphy gives the conclusion, including at singular points.
\end{proof}

\begin{lemma}[Hartogs extension from a boundary neighborhood]
\label{lem:boundary-hartogs}
Let $Y$ be a compact complex manifold, let $M\subset Y$ be compact, and
let $\Omega$ be a connected component of $Y\setminus M$. Assume that
$\Omega$ is Stein and $\dim_\CC\Omega\ge2$. Let $N$ be an open
neighborhood of $M$ in $Y$ and let
\[
g\in\Mcal(N\cap\Omega).
\]
Then there are a smaller neighborhood $N'\subset N$ of $M$ and a
meromorphic function
\[
\widetilde g\in\Mcal(\Omega)
\]
such that
\[
\widetilde g=g\qquad\text{on }N'\cap\Omega.
\]
\end{lemma}

\begin{proof}
Since $Y$ is compact and $N$ contains $\partial\Omega\subset M$, the set
$\Omega\setminus N$ is relatively compact in $\Omega$. A connected
Stein manifold of complex dimension at least two has only one end; see,
for instance, Serre \cite{SerreDuality}. Hence there exists a compact set
$K\Subset\Omega$ containing $\Omega\setminus N$ such that
\[
\Omega\setminus K
\]
is connected. In particular,
\[
\Omega\setminus K\subset N\cap\Omega.
\]
The meromorphic Hartogs theorem
\cite[Theorem~2.2]{MerkerPorten} extends
$g|_{\Omega\setminus K}$ meromorphically to all of $\Omega$; denote the
extension by $\widetilde g$.

Since $K$ is compact and disjoint from $M$, we may shrink $N$ to a
neighborhood $N'\subset N$ of $M$ such that $N'\cap K=\varnothing$.
Then
\[
N'\cap\Omega\subset\Omega\setminus K,
\]
so $\widetilde g=g$ on $N'\cap\Omega$, as claimed.
\end{proof}

\section{Projective globalization}
\label{sec:globalization}

We now prove the higher-dimensional projective criterion by reducing
to a generic linear two-plane. The plane can be chosen to avoid the
analytic exceptional locus $E_D=D\setminus D^\circ$, while meeting
$D^\circ$ transversely. On that section the normalized transverse
parameters are automatically compatible. Local pseudoconvexity,
Takeuchi's theorem \cite[Theorem~IV, p.~181]{Takeuchi}, and
Lemma~\ref{lem:boundary-hartogs} then globalize the resulting first
integral on the plane.

\begin{proof}[Proof of Theorem~\ref{thm:projective-main}]
Let $p\in D^\circ$ be the pencil singularity in hypothesis
\textup{(iii)}. By Proposition~\ref{prop:nearMreg}, the corresponding
minimal local parameter extends by holonomy to a single-valued
meromorphic first integral
\[
G\in\Mcal(N_{\reg})
\]
on an open neighborhood $N_{\reg}$ of $M_{\reg}$. By
Lemma~\ref{lem:monodromy}, the projective transverse monodromy along
$D^\circ$ is trivial; hence the local minimal first integrals along
$D^\circ$ may be normalized as $\widehat H_q$, $q\in D^\circ$, so as
to agree with $G$ wherever both are defined.

By the analytic exceptional-locus condition in Definition~\ref{def:generic-ltp},
we have
\[
E_D=D\setminus D^\circ,
\qquad \dim_\CC E_D\le n-3.
\]
Choose a generic linear plane
\[
P\simeq\PP^2\subset\PP^n
\]
such that
\begin{enumerate}[label=\textup{(\alph*)}]
\item $P\cap E_D=\varnothing$ and $P$ is transverse to $D^\circ$;
\item $P$ is transverse to $M_{\reg}$ at every point of
$P\cap M_{\reg}$;
\item $P$ is in general position with $\F$: it is not invariant and
its tangency locus with $\F$ has codimension at least two in $P$;
\item the restriction $\F_P:=\F|_P$ is a codimension-one foliation on
$P$, and at every point of $P\cap(M\cap\Sing(\F))$ outside $D$ the
restriction to $P$ of a local holomorphic first integral is nonconstant.
\end{enumerate}
Here is a direct justification of condition \textup{(b)}. Let
$\mathcal G=\operatorname{Gr}(3,n+1)$ be the Grassmannian of projective
two-planes, of real dimension $6n-12$. At a fixed point
$x\in M_{\reg}$, nontransversality is equivalent to
\[
T_xP\subset T_x^{\CC}M,
\]
since $T_xP$ is complex and $T_x^{\CC}M$ is the maximal complex
subspace of the real hyperplane $T_xM$. The tangent directions of
such planes through $x$ form $\operatorname{Gr}(2,n-1)$, of real
dimension $4(n-3)$. The incidence space of nontransverse pairs
$(x,P)$ therefore has real dimension
\[
(2n-1)+4(n-3)=6n-13<\dim_{\RR}\mathcal G.
\]
By Sard's theorem its image has measure zero. Thus a plane can be
chosen transverse to $M_{\reg}$ at every point of intersection.
Condition \textup{(c)} holds on a Zariski open dense set of planes
\cite[Theorem~3.25]{CerveauLinsNetoStructural}; the remaining
conditions are obtained by the usual generic linear-section choices.
In particular, these requirements can be imposed simultaneously. Then
\[
P\cap D=\{q_1,\ldots,q_s\}\subset D^\circ.
\]
Put
\[
M_P=M\cap P,
\qquad G_P=G|_{P\cap N_{\reg}}.
\]
The function $G_P$ is a meromorphic first integral for $\F_P$ near
$P\cap M_{\reg}$.

At each $q_j$, transversality of $P$ to $D$ identifies the germ of
$\F_P$ with the fixed two-dimensional transverse germ of the
local-product structure. Consequently
\[
H_j:=\widehat H_{q_j}|_P
\]
is a minimal local meromorphic first integral of $\F_P$ at $q_j$.
By the normalization along $D^\circ$, $H_j$ agrees with $G_P$ on a
nonempty open subset of their common domain. Hence $H_j$ extends
$G_P$ meromorphically through $q_j$. In particular, the several
dicritical points of the plane section are automatically compatible;
with the inherited normalization all transition maps are the identity.

Let now $a\in M_P\cap\Sing(\F_P)$ with $a\notin P\cap D$.
Condition \textup{(b)} excludes tangencies of $P$ with the regular ambient
foliation at points of $M_{\reg}$, because
$T_a\F=T_a^{\CC}M$ there. Hence such a point $a$ lies in
$M\cap\Sing(\F)$ outside $D$, and is nondicritical by hypothesis
\textup{(ii)} of Theorem~\ref{thm:projective-main}. By the genericity
condition \textup{(d)}, a local holomorphic first integral restricts
nontrivially to $P$. The two-dimensional form of
Lemma~\ref{lem:nondicritical} therefore extends $G_P$ through $a$.
Thus $G_P$ extends to a nonconstant meromorphic first integral on a
neighborhood $N_P$ of $M_P$ in $P$.

Let $\Omega$ be a connected component of $P\setminus M_P$.
Proposition~\ref{prop:pseudoconvex} shows that $\Omega$ is locally
pseudoconvex in $P$, and Takeuchi's theorem
\cite[Theorem~IV, p.~181]{Takeuchi} implies that $\Omega$ is Stein.
Lemma~\ref{lem:boundary-hartogs}, applied to
$G_P|_{N_P\cap\Omega}$, extends $G_P$ meromorphically through $\Omega$
after shrinking $N_P$ near $M_P$. Repeating this for the finitely many
components of $P\setminus M_P$ and shrinking once more if necessary,
the extensions glue with $G_P$ on a neighborhood of $M_P$. We obtain a
global meromorphic function on $P\simeq\PP^2$, hence a nonconstant
rational first integral of $\F_P$.

Finally, the extension statement in
\cite[Theorem~3.25]{CerveauLinsNetoStructural} applies to the plane
$P$, which is in general position with $\F$ by condition
\textup{(c)}: the rational first integral of $\F_P$ extends to a
rational first integral of $\F$ on $\PP^n$. This completes the proof.
\end{proof}

\begin{proof}[Proof of Corollary~\ref{cor:local-separation}]
Let $\gamma_p\subset\PP^1$ be the set of parameters of the local Levi
leaves represented in $M_{\reg}\cap U_p$. Since $h$ is constant on local Levi leaves
and the generic fibers of the minimal parameter $F_p$ are connected,
it induces a map $\eta:\gamma_p\longrightarrow S$ with
\[
h=\eta\circ F_p
\]
on generic local leaves. The assumption that $h$ separates the nearby
Levi leaves means that $\eta$ is injective on the entire relevant
transverse parameter set, including when it is a circle.

Let a generic global Levi leaf $L$ have two sufficiently near-$p$
components contained respectively in $V_{\lambda_1}$ and
$V_{\lambda_2}$. By the global-leaf compatibility assumed for $h$,
\[
\eta(\lambda_1)=h|_{L\cap U_p}=\eta(\lambda_2).
\]
Hence $\lambda_1=\lambda_2$. By
Lemma~\ref{lem:product-propagation}, the assumption that each Levi leaf
accumulates at some point of the local-product locus gives
\[
p\in\overline L
\]
for every Levi leaf $L\subset M_{\reg}$. Together with the uniqueness of
the local parameter proved above, this shows that $p$ is a pencil
singularity for $M$. Theorem~\ref{thm:projective-main} now gives the
conclusion.
\end{proof}

\section{Real-analytic transverse foliations}
\label{sec:real-transverse}

We consider the interesting case where  the
holomorphic foliation is contained in a real-analytic foliation of real
codimension one. We discover that the induced real foliation on a complex transversal
restricts the possible projective monodromy.

\subsection{Singular real Levi foliations}
\label{subsec:singular-real-levi}

Let $X$ be a connected complex manifold, let $\F$ be a codimension-one
holomorphic foliation on $X$, and let $\Hcal$ be a possibly singular
real-analytic foliation of real codimension one. We write
\[
\F\subset\Hcal
\]
if, wherever both foliations are regular, every leaf of $\F$ is
contained in a leaf of $\Hcal$. On the common regular locus this is
equivalent to
\[
T_{\RR}\F\subset T\Hcal.
\]

\begin{lemma}
\label{lem:real-transverse}
Assume that $\F\subset\Hcal$. Then  the holonomy pseudogroup of $\mathcal F$ acting on the transverse parameter defined by $\F$
 preserve the real one-dimensional foliation induced by
$\Hcal$ on a complex transversal. In particular, whenever minimal
local meromorphic first integrals identify the transverse leaf space
with $\PP^1$, the corresponding projective monodromy is contained in
the subgroup of $\operatorname{PGL}(2,\CC)$ preserving that real
foliation.
\end{lemma}

\begin{proof}
At a regular point of $\F$, choose a local holomorphic first integral
$f_i$ on a flow box. Since $\F\subset\Hcal$, the foliation $\Hcal$
descends to a real one-dimensional foliation on the complex
transversal parametrized by $f_i$. On an overlap,
\[
f_j=\varphi_{ij}(f_i)
\]
for a holomorphic change of transverse coordinate. Since $\Hcal$ is
globally defined, $\varphi_{ij}$ sends leaves of the induced real
foliation in the $f_i$-plane to leaves of the induced real foliation
in the $f_j$-plane. This proves the first assertion.

When minimal meromorphic first integrals identify the local transverse
leaf spaces with $\PP^1$, the corresponding projective transition
maps therefore belong to the subgroup of
$\operatorname{PGL}(2,\CC)$ preserving the induced real foliation.
\end{proof}

\section{Infinitely many invariant Levi-flats and holonomy rigidity}
\label{sec:many-leviflats-holonomy}

We point-out a local consequence of the simultaneous presence of many
invariant real-analytic Levi-flat varieties.  The point is purely
transverse: infinitely many distinct real-analytic curves in a complex
transversal force the common holonomy group to be abelian.  No
compactness assumption on the ambient manifold and no nonsingularity
assumption on the foliation are involved.

We begin with the one-dimensional statement.

\begin{lemma}[Transverse rigidity]
\label{lem:transverse-rigidity}
Let
\[
G\subset\operatorname{Diff}(\CC,0)
\]
be a group of germs of holomorphic diffeomorphisms. Assume that there
exists an infinite family
\[
\{\gamma_\alpha\}_{\alpha\in A}
\]
of pairwise distinct smooth real-analytic curve germs through the
origin such that
\[
g(\gamma_\alpha)=\gamma_\alpha
\]
for every $g\in G$ and every $\alpha\in A$. Then $G$ is abelian.
\end{lemma}

\begin{proof}
For a smooth real-analytic curve germ $\gamma\subset(\CC,0)$, denote by
$\sigma_\gamma$ the Schwarz reflection across $\gamma$.  Thus
$\sigma_\gamma$ is the antiholomorphic involution fixing $\gamma$
pointwise.  If $\phi$ is a local holomorphic coordinate straightening
$\gamma$ to $\RR$, then
\[
\sigma_\gamma=\phi^{-1}\circ c\circ\phi,
\qquad c(z)=\overline z.
\]
If $g\in\operatorname{Diff}(\CC,0)$ preserves $\gamma$, then
$g\sigma_\gamma g^{-1}$ is again the antiholomorphic involution fixing
$\gamma$ pointwise.  By uniqueness of Schwarz reflection,
\[
g\circ\sigma_\gamma=\sigma_\gamma\circ g.
\tag{\ref{lem:transverse-rigidity}.1}
\]
Hence every element of $G$ commutes with every reflection associated
with the family.

We distinguish two cases.

Suppose first that the set of tangent lines
\[
\{T_0\gamma_\alpha:\alpha\in A\}
\]
is infinite. Since any $g\in G$ preserves one such real tangent line,
its multiplier is real:
\[
g'(0)\in\RR^*.
\]
If some $f\in G$ has $|f'(0)|\ne1$, then $f$ is analytically
linearizable. In a linearizing coordinate we may write
\[
f(z)=\lambda z,
\qquad \lambda\in\RR^*,\quad |\lambda|\ne1.
\]
Every smooth real-analytic $f$-invariant curve through the origin is a
straight real line. Indeed, after a rotation, and after replacing $f$
by a suitable iterate or inverse, write the curve as
\[
y=\varphi(x),\qquad 0<\lambda<1,
\]
with $\varphi(0)=\varphi'(0)=0$. Invariance gives
\[
\varphi(\lambda x)=\lambda\varphi(x).
\]
Writing $\varphi(x)=\sum_{j\ge2}a_jx^j$ yields
$a_j(\lambda^j-\lambda)=0$, and therefore $\varphi\equiv0$.
Consequently our family contains infinitely many distinct real lines
\[
L_\theta=e^{i\theta}\RR.
\]
If
\[
g(z)=\sum_{j\ge1}b_jz^j\in G,
\]
then $g(L_\theta)=L_\theta$ implies
\[
b_je^{i(j-1)\theta}\in\RR
\]
for every $j$ and infinitely many values of $\theta$ modulo $\pi$.
For $j\ge2$ this forces $b_j=0$. Hence every element of $G$ is a real
homothety, and $G$ is abelian.

It remains, in the present case, to assume that $|g'(0)|=1$ for every
$g\in G$. Since all multipliers are real,
\[
g'(0)\in\{1,-1\}.
\]
If $g'(0)=1$ and $g\ne\operatorname{id}$, write
\[
g(z)=z+az^{k+1}+O(z^{k+2}),
\qquad a\ne0.
\]
Let $\gamma$ be a smooth real-analytic invariant curve with tangent
vector $v\ne0$, and write
\[
\gamma(t)=vt+O(t^2),
\qquad
g(\gamma(t))=\gamma(\rho(t)).
\]
Comparison at the first nonlinear order gives
\[
av^k\in\RR.
\]
Thus a nontrivial germ tangent to the identity admits invariant smooth
real-analytic curves in only finitely many tangent directions, contrary
to the present hypothesis. Hence $g'(0)=1$ implies
$g=\operatorname{id}$. If $g'(0)=-1$, then $(g^2)'(0)=1$, so
$g^2=\operatorname{id}$. Moreover, if $g$ and $h$ both have multiplier
$-1$, then $(gh)'(0)=1$, whence $gh=\operatorname{id}$ and $g=h$.
Thus $G$ has at most two elements and is again abelian.

Suppose now that only finitely many tangent directions occur. Since the
family is infinite, there are two distinct curves $\gamma_1$ and
$\gamma_2$ with
\[
T_0\gamma_1=T_0\gamma_2.
\]
Let $\sigma_1$ and $\sigma_2$ be their Schwarz reflections and set
\[
F=\sigma_2\circ\sigma_1.
\]
Then $F$ is holomorphic and $F'(0)=1$. Moreover $F$ is not the
identity, since otherwise $\sigma_1=\sigma_2$ and therefore
$\gamma_1=\gamma_2$. Hence
\[
F(z)=z+az^{k+1}+\cdots,
\qquad a\ne0.
\]
By \textup{(\ref{lem:transverse-rigidity}.1)}, every element of $G$
commutes with both $\sigma_1$ and $\sigma_2$, and therefore
\[
G\subset Z(F),
\]
where $Z(F)$ denotes the centralizer of $F$ in
$\operatorname{Diff}(\CC,0)$.

We recall the standard formal argument showing that $Z(F)$ is abelian.
There is a unique formal vector field
\[
X=\log F
\]
with $F=\exp X$. If $gF=Fg$, then
\[
g_*X=X.
\]
Writing
\[
X=a(z)\frac{\partial}{\partial z},
\]
consider the formal meromorphic differential
\[
\omega=\frac{dz}{a(z)}.
\]
The identity $g_*X=X$ is equivalent to $g^*\omega=\omega$. If $T$ is a
formal primitive of $\omega$, then for every $g\in Z(F)$ one has
\[
T\circ g=T+c_g
\]
for some constant $c_g$. Therefore, for $g,h\in Z(F)$,
\[
T\circ g\circ h=T+c_h+c_g
=T+c_g+c_h=T\circ h\circ g.
\]
Since $T$ is a nonconstant formal Fatou coordinate, this implies
$g\circ h=h\circ g$. Thus $Z(F)$, and hence $G$, is abelian.
\end{proof}

The transverse lemma applies directly to holonomy.

The following elementary formal consequence will also be used below.
The following lemma is well-known (\cite{Paul1999}) but we give a brief proof.
\begin{lemma}[Invariant formal transverse differential]
\label{lem:abelian-invariant-form}
Let
\[
G\subset\operatorname{Diff}(\CC,0)
\]
be an abelian group. Then there exists a nonzero formal meromorphic
$1$-form $\widehat\omega$ at the origin such that
\[
g^*\widehat\omega=\widehat\omega
\qquad\text{for every }g\in G.
\]
\end{lemma}

\begin{proof}
If $G$ is formally linearizable, choose a formal coordinate $z$ in
which every element is a homothety. Then
\[
\widehat\omega=\frac{dz}{z}
\]
is invariant by every element in $G$. Otherwise $G$ contains a non-periodic element tangent to
the identity after taking a suitable iterate. Write this element as
\[
g_0=\exp \hat {\mathcal Z},
\qquad
\hat {\mathcal Z}=a(z)\frac{\partial}{\partial z},
\]
for some formal one-variable vector field $\hat {\mathcal Z}$ with $\hat {\mathcal Z}\ne0$. Since $G$ is abelian, every  $g \in G$ commutes with
$g_0$, hence preserves the unique formal logarithm $\hat {\mathcal Z}$. Therefore it
preserves
\[
\widehat\omega=\frac{dz}{a(z)}.
\]
\end{proof}

As a follow-up to the lemma above we have:

\begin{lemma}[Closed rational forms without saddle-nodes are logarithmic]
\label{lem:closed-rational-logarithmic}
Let $\mathcal F$ be a holomorphic foliation on $\mathbb P^2$ defined by a
closed rational $1$-form. Assume that the reduction of singularities of
$\mathcal F$ contains no saddle-nodes. Then $\mathcal F$ is defined by a
closed logarithmic rational $1$-form.
\end{lemma}

\begin{proof}
Let $\omega$ be a closed rational $1$-form defining $\mathcal F$. Write
\[
\omega=\alpha+dR,
\]
where $\alpha$ is logarithmic and
\(
R\in\mathbb C(\mathbb P^2).
\)
If $R$ is constant, there is nothing to prove. Also, if $\mathcal F$
admits a rational first integral $H$, then
\[
\frac{dH}{H}
\]
is a closed logarithmic $1$-form defining the same foliation. We may
therefore assume, seeking a contradiction, that $R$ is nonconstant and
that $\mathcal F$ admits no rational first integral.

Start with a reduction of singularities of $\mathcal F$. By hypothesis
all its reduced singularities are nondegenerate. Further point blow-ups
do not create saddle-nodes: blowing up a regular point produces at worst
a reduced nondegenerate singularity, and blowing up a reduced
nondegenerate singularity whose eigenvalue quotient is not a positive
rational again produces only reduced nondegenerate singularities. Thus
the absence of saddle-nodes persists under any additional blow-ups.

We may therefore resolve simultaneously the remaining indeterminacy of
$R$ and the polar divisor of $\alpha$:
\[
\pi:S\longrightarrow\mathbb P^2,
\]
in such a way that
\[
f:=R\circ\pi:S\longrightarrow\mathbb P^1
\]
is a holomorphic map and the relevant divisors have normal crossings.
We have
\[
\pi^*\omega=\widetilde\alpha+df,
\qquad
\widetilde\alpha=\pi^*\alpha,
\]
with $\widetilde\alpha$ logarithmic.

We claim that every irreducible component of the polar divisor of
$\widetilde\alpha$ is vertical with respect to $f$. Suppose, on the
contrary, that $H$ is a horizontal polar component. Then $H$ meets the
fiber
\[
f^{-1}(\infty).
\]
After the above resolution, choose a point
\[
p\in H\cap f^{-1}(\infty)
\]
at which $H$ meets a unique irreducible component of that fiber. In
suitable local coordinates centered at $p$,
\[
f=x^{-m},
\qquad
H=\{y=0\},
\qquad m\geq1,
\]
and
\[
\widetilde\alpha=
\lambda\frac{dy}{y}
+\mu\frac{dx}{x}
+\beta,
\qquad
\lambda\neq0,
\]
where $\beta$ is holomorphic. Hence
\[
\pi^*\omega
=
-mx^{-m-1}\,dx
+\lambda\frac{dy}{y}
+\mu\frac{dx}{x}
+\beta.
\]
Multiplying by $x^{m+1}y$, which does not change the foliation, gives a
holomorphic $1$-form of the shape
\[
\Omega=
\bigl(-my+\text{higher order terms}\bigr)\,dx
+
\bigl(\lambda x^{m+1}+\text{higher order terms}\bigr)\,dy.
\]
Its associated vector field has one zero eigenvalue and one nonzero
eigenvalue. Thus $p$ is a saddle-node, contradicting the hypothesis.
Therefore all polar components of $\widetilde\alpha$ are vertical.

Let
\[
\operatorname{Res}(\widetilde\alpha)
=
\sum_a D_a,
\]
where each divisor $D_a$ is supported on the fiber $f^{-1}(a)$. Since
$\widetilde\alpha$ is a global closed logarithmic form, its residue
divisor has trivial Chern class. Therefore, for every irreducible
component $C$ of $f^{-1}(a)$,
\[
D_a\cdot C=0.
\]
By Zariski's lemma, the intersection matrix of the irreducible
components of a fiber is negative semidefinite and its kernel is
generated by the multiplicity vector of the whole fiber. Consequently,
for every $a$,
\[
D_a=c_a\,f^{-1}(a)
\]
for some $c_a\in\mathbb C$.

Since all fibers of $f$ are numerically equivalent and the total
residue divisor has trivial Chern class, one also has
\[
\sum_a c_a=0.
\]
Hence there exists a logarithmic $1$-form $\beta_0$ on $\mathbb P^1$
with residue divisor
\[
\sum_a c_a[a].
\]
Then
\[
\widetilde\alpha-f^*\beta_0
\]
has zero residues and hence is holomorphic on $S$. Since $S$ is a
rational surface,
\[
H^0(S,\Omega_S^1)=0,
\]
and therefore
\[
\widetilde\alpha=f^*\beta_0.
\]
It follows that
\[
\pi^*\omega
=
df+f^*\beta_0
=
f^*(dt+\beta_0).
\]
Hence the transformed foliation is tangent to the fibers of $f$, so
$f$ is a rational first integral of $\pi^*\mathcal F$. Thus $R$ is a
rational first integral of $\mathcal F$, contradicting our assumption.

Therefore either $R$ is constant, in which case $\omega$ is
logarithmic, or $\mathcal F$ admits a rational first integral, in which
case it is defined by a logarithmic form of the type $dH/H$. In either
case $\mathcal F$ is logarithmic.
\end{proof}

We shall use the following notion. 
\begin{definition}[virtual holonomy, \cite{CamachoLinsNetoSad}]
Let $L$ be a leaf of $\mathcal F$, let $q\in L$, and let
$\Sigma$ be a small transversal to $\mathcal F$ through $q$.
The virtual holonomy group of $L$ with respect to $\Sigma$ is
\[
\operatorname{Hol}^{\mathrm{virt}}(\mathcal F,L,\Sigma,q)
=
\left\{
h\in\operatorname{Diff}(\Sigma,q)\,;\,
L_z=L_{h(z)}
\ \text{for every }z\in(\Sigma,q)
\right\},
\]
where $L_z$ denotes the leaf of $\mathcal F$ through $z$.
Equivalently, it consists of the germs of local transverse
biholomorphisms fixing $q$ and preserving each leaf of the foliation.
\end{definition}
The usual holonomy group is naturally a subgroup of the virtual
holonomy group.

\begin{proposition}[Virtual holonomy rigidity from many Levi-flats]
\label{prop:many-leviflats-holonomy}
Let $\F$ be a codimension-one holomorphic foliation on a complex
manifold $X$, let $L$ be a regular leaf, and fix $q\in L$. Assume that
there is an infinite family $\{M_\alpha\}_{\alpha\in A}$ of invariant
real-analytic Levi-flat variety germs, smooth along $L$, pairwise distinct
near $q$, such that $L\subset M_\alpha$ for every $\alpha$. Then the
virtual holonomy group of $L$ at $q$ is abelian. In particular, the
ordinary holonomy group is abelian.
\end{proposition}

\begin{proof}
Choose a sufficiently small complex transversal
\[
(\Sigma,q)\simeq(\CC,0)
\]
to $L$. Since each $M_\alpha$ is real analytic, Levi-flat and saturated
by $\F$, its intersection with $\Sigma$ is a smooth real-analytic curve
germ
\[
\gamma_\alpha=M_\alpha\cap\Sigma.
\]
In a foliated flow box a saturated hypersurface germ is completely
determined by its transverse trace. Hence the pairwise distinct germs
$M_\alpha$ give pairwise distinct germs $\gamma_\alpha$.

Let $h$ belong to the virtual holonomy group. By definition, $h(z)$ is
on the same local leaf as $z$. Since every $M_\alpha$ is saturated, a
point $z\in\gamma_\alpha$ and every point on its local leaf belong to
$M_\alpha$. Hence
\[
h(\gamma_\alpha)=\gamma_\alpha
\]
for every $\alpha$. Lemma~\ref{lem:transverse-rigidity} applies to the
whole virtual holonomy group and proves that it is abelian.
\end{proof}

\begin{remark}
\label{rem:many-leviflats-local}
The proposition is local along the leaf. In particular, it does not
require the ambient manifold to be compact, nor does it require the
foliation to be nonsingular away from the chosen regular leaf. The
essential hypotheses are the existence of infinitely many distinct
invariant real-analytic Levi-flat germs containing the same leaf and
the resulting common family of invariant transverse curves.
\end{remark}

\section{Darboux theorem for Levi-flats}
\label{sec:many-leviflats-plane}

We now combine the preceding transverse rigidity with elementary
intersection theory on the projective plane.  Throughout this section
we assume that 
\[
\Sing(M_\alpha)\subset\Sing(\F)
\]
for all the invariant real-analytic codimension one irreducible varieties  under consideration.

\begin{proof}[Proof of Theorem~\ref{thm:leviflat-intersection-alternative}]
Fix infinitely many members $M_\alpha$ satisfying \textup{(H)} and put
\[
X_\alpha=(M_0\cap M_\alpha)\setminus\Sing(\F).
\]
Let $q\in X_\alpha$. Since $q\notin\Sing(\F)$ and, by the standing
condition, both hypersurfaces are smooth at $q$, choose a holomorphic
flow box
\(
(z,t)\in\Delta\times\Delta
\)
in which
\(
\F=\{t=\text{constant}\}.
\)
Since $M_0$ and $M_\alpha$ are invariant, after shrinking the box they
have the form
\[
M_0=\{t\in\gamma_0\},\qquad
M_\alpha=\{t\in\gamma_\alpha\},
\]
where $\gamma_0$ and $\gamma_\alpha$ are smooth real-analytic curve
germs in the transverse disc.

The two transverse curve germs cannot coincide. Indeed, equality of
the transverse germs would imply equality of the two saturated
hypersurface germs in the flow box; since $M_0$ and $M_\alpha$ are
 irreducible real-analytic variety of pure real codimension one, analytic continuation would
then give $M_0=M_\alpha$, contrary to the hypothesis. Hence
$\gamma_0\cap\gamma_\alpha$ is discrete. It follows that $X_\alpha$ is
locally a union of plaques $\{t=t_j\}$ and therefore is a
one-dimensional complex analytic subset of
\(
\PP^2\setminus\Sing(\F).
\)
Every irreducible component is invariant by $\F$.

The singular set of a foliation on $\PP^2$ is finite. Since
$M_0\cap M_\alpha$ is closed in $\PP^2$, an irreducible component of
$X_\alpha$ can acquire boundary, in passing to $\PP^2$, only at points
of $\Sing(\F)$. The Remmert--Stein extension theorem therefore shows
that its closure is a one-dimensional complex analytic subset of
$\PP^2$. By Chow's theorem this closure is algebraic. Thus, for every
relevant $\alpha$, choose an irreducible invariant algebraic curve
\[
C_\alpha\subset M_0\cap M_\alpha.
\]

If infinitely many distinct curves $C_\alpha$ occur, then $\F$ has
infinitely many invariant algebraic curves. By the Darboux--Jouanolou
theorem \cite{Jouanolou}, $\F$ admits a nonconstant rational first
integral; this is alternative \textup{(i)}.

Assume instead that only finitely many distinct curves $C_\alpha$
occur. By the pigeonhole principle, one of them, say $C$, is contained
in infinitely many members $M_{\alpha_j}$ of the family. Let $L$ be a
connected regular leaf component of
\[
C\setminus\Sing(\F),
\]
and choose a small complex transversal $(\Sigma,q)\simeq(\CC,0)$ to
$L$. Each $M_{\alpha_j}$ determines a smooth real-analytic transverse
curve germ
\[
\gamma_j=M_{\alpha_j}\cap\Sigma.
\]
Because $M_{\alpha_j}$ is saturated by $\F$, every element of the
virtual holonomy group of $L$ preserves $\gamma_j$. Moreover these
transverse germs are pairwise distinct: equality of two of them would
imply equality of the corresponding saturated hypersurface germs, and
irreducibility and real-analytic continuation would then imply equality
of the global Levi-flats. Proposition~\ref{prop:many-leviflats-holonomy}
therefore gives that the virtual holonomy group of $L$ is abelian.
This proves alternative \textup{(ii)} and the virtual-holonomy
assertion.
\end{proof}

Now we shall prove the corollary.

\begin{proof}[Proof of Corollary~\ref{corollary:abelian}]
Assume now the additional hypotheses in the last paragraph of
Theorem~\ref{thm:leviflat-intersection-alternative}. On every connected
component $L$ of $C\setminus\Sing(\F)$, the virtual holonomy group is
abelian. By Lemma~\ref{lem:abelian-invariant-form}, on a transverse disc
there is a nonzero invariant formal meromorphic differential. Suspending
it along $L$ gives a closed formal meromorphic $1$-form
\[
\widehat\eta_L
\]
on the formal neighborhood of $L$, defining $\F$ there. Holonomy
invariance is precisely the condition making the suspension
single-valued.

Let
\[
\pi:S=\operatorname{Bl}_{p_1,\ldots,p_m}\PP^2\longrightarrow\PP^2
\]
be the blow-up of the dicritical singularities occurring in the theorem,
and let $\widetilde C$ be the strict transform of $C$. Pull back the
formal form to the regular part of $\widetilde C$. By hypothesis, every
singularity of the transformed foliation on $\widetilde C$ lying over a
dicritical point is reduced, non-dicritical, and is not a saddle-node;
the other singularities were already reduced and non-saddle-node. Since $C$ is
contained in infinitely many invariant real-analytic Levi-flat germs,
the same is true after pull-back along $\widetilde C$. The local theorem
of Cerveau--Lins Neto \cite{CLN}, or Brunella's formulation
\cite{Brunella}, supplies a nonconstant local meromorphic first integral
at these singular points. In the reduced non-dicritical non-saddle-node
case this is the analytically linearizable resonant situation; hence in
suitable local coordinates one has a primitive first integral
\[
H=x^n y^m,\qquad n,m\ge1,
\]
and a closed logarithmic defining form
\[
\eta_0=\frac{dH}{H}
=n\frac{dx}{x}+m\frac{dy}{y}.
\]
On the punctured formal neighborhood of the branch of $\widetilde C$
approaching such a point,
\[
\widehat\eta_L=\Phi\,\eta_0.
\]
Closedness gives $d\Phi\wedge\eta_0=0$, so $\Phi$ is a formal
meromorphic first integral and therefore $\Phi=\varphi(H)$ for a
one-variable formal meromorphic germ $\varphi$. Since $H=x^ny^m$
vanishes on the invariant axes, $\varphi(H)$ is formal meromorphic at
the origin. Thus the form extends formally meromorphically through the
point.

We must also check compatibility when two invariant components meet at
a corner. Let $q$ be such a point and let $\widehat\eta_1$ and
$\widehat\eta_2$ be the closed formal meromorphic forms obtained from
the two sides. Since they define the same foliation, on their common
formal domain one has
\[
\widehat\eta_1=\Psi\,\widehat\eta_2,
\]
where closedness implies
\[
d\Psi\wedge\widehat\eta_2=0.
\]
Hence $\Psi$ is a formal meromorphic first integral. If the germ has no
nonconstant meromorphic first integral, then $\Psi$ is constant and the
two forms glue after a constant normalization. The only remaining
possibility at a reduced non-dicritical corner is the resonant
integrable case, analytically conjugate to
\[
d(x^n y^m)=0.
\]
There the finite-to-finite Dulac correspondence identifies the
transverse dynamics of the two components. Equivalently, one adjoins the
holonomy groups of the adjacent components through the Dulac
correspondence. By construction, this adjoined holonomy group is still
contained in the virtual holonomy group of the leaf under consideration.
Since the latter is abelian in our situation, the adjoined group is
abelian as well. Hence the invariant transverse formal differential is
the same on both sides, up to the same constant normalization, and the
suspended closed forms are compatible. This is the standard adjunction
mechanism for holonomy groups along a reducible invariant divisor; see
Camacho--Lins Neto--Sad~\cite{CamachoLinsNetoSad},
Paul~\cite{Paul1999}, Camacho--Sc\'ardua~\cite{CamachoScarduaLiouvillian},
and Sc\'ardua~\cite{ScarduaBoundedGeometry}.

Consequently the local extensions agree through all corners and all
singular points of $\widetilde C$. Performing this at all points of
$\widetilde C$ gives a global closed formal meromorphic $1$-form
\[
\widehat\eta
\]
on the formal completion $\widehat S_{\widetilde C}$, defining the
transformed foliation formally along $\widetilde C$.

We now use the numerical hypothesis. Write
\[
d=\deg C,\qquad d_j=\operatorname{mult}_{p_j}(C).
\]
On the blown-up surface,
\[
\widetilde C^{\,2}=d^2-\sum_{j=1}^m d_j^2.
\]
Since
\(
\sum_{j=1}^m d_j<d,
\)
we have
\[
\sum_{j=1}^m d_j^2
\le \left(\sum_{j=1}^m d_j\right)^2<d^2,
\]
and therefore
\(
\widetilde C^{\,2}>0.
\)
This is the numerical condition used by Sad in his treatment of
controlled dicritical singularities; radial singularities are among the
allowed examples \cite{Sad1995}.

Hironaka--Matsumura, Proposition~3.6 in \cite{HM}, states that if $D$ is
an effective divisor on a nonsingular projective surface, $D^2>0$, and
every prime component of $D$ has nonnegative self-intersection, then the
field of formal-rational functions along $\operatorname{Supp}(D)$ equals
the rational function field of the surface. Applying this to the
irreducible divisor $D=\widetilde C$ gives
\[
K\!\left(\widehat S_{\widetilde C}\right)=K(S).
\]
Choose rational coordinates $(u,v)$ on a dense affine chart of $S$ and
write
\[
\widehat\eta=\widehat A\,du+\widehat B\,dv.
\]
The coefficients belong to
$K(\widehat S_{\widetilde C})=K(S)$, hence they are the formal germs of
rational functions $A,B\in K(S)$. Therefore
\[
\eta_S=A\,du+B\,dv
\]
is a rational $1$-form on $S$ whose formal completion along
$\widetilde C$ is $\widehat\eta$. Injectivity of the completion map
shows that
\[
d\eta_S=0
\]
and, if $\omega_S$ is a rational defining form of the transformed
foliation, that
\[
\eta_S\wedge\omega_S=0.
\]
Thus $\eta_S$ is a closed rational defining form for the transformed
foliation.

Finally, $\pi$ is birational, so
\[
K(S)=K(\PP^2)
\]
and hence
\[
\Omega^1_{K(S)/\CC}=\Omega^1_{K(\PP^2)/\CC}.
\]
Thus $\eta_S$ is the pull-back, in the rational category, of a unique
rational $1$-form $\eta$ on $\PP^2$. The identities above descend to
\[
d\eta=0,
\qquad
\eta\wedge\omega_{\F}=0,
\]
for any rational defining form $\omega_{\F}$ of $\F$. Consequently
$\eta$ is a nonzero closed rational $1$-form defining $\F$.
\end{proof}

\begin{remark}
\label{rem:H-branch}
Hypothesis \textup{(H)} isolates the branch in which one Levi-flat meets
infinitely many others at regular points of the foliation. It is not
automatic. For a rational pencil one may choose infinitely many
pairwise disjoint real-analytic curves in the parameter sphere; their
inverse images are invariant Levi-flats which may meet only at the base
locus. The next result addresses a model for this separated branch.
\end{remark}

\subsection{Separated pencil-type families}

\begin{theorem}[Separated pencil-type Levi-flats]
\label{thm:separated-leviflats-closed-form}
Let $\F$ be a holomorphic foliation on $\PP^2$, let
\[
p\in\Sing(\F),
\]
and let
\[
F_p:(U,p)\dashrightarrow\PP^1
\]
be a minimal local meromorphic first integral. Assume that $\F$ admits infinitely many pairwise distinct irreducible
global invariant real-analytic Levi-flat varieties $M_j$ passing through
$p$ and such that
\[
(M_j\cap M_k)\cap(U\setminus\{p\})=\varnothing,
\qquad j\neq k.
\]
Assume finally that the minimal transverse parameter admits projectively
compatible analytic continuation throughout $\PP^2\setminus\Sing(\F)$,
in the sense that the changes of minimal parameter are M\"obius
transformations and that, for each $M_j$, one connected component of the
local transverse trace associated with a branch of $M_j$ at $p$ continues
to a compact real-analytic subset $\gamma_j\subset\PP^1$ of pure real
dimension one. Assume moreover that the continued sets $\gamma_j$ are
pairwise disjoint and that each of them is preserved individually by the
projective monodromy. Then $\F$ is defined on $\PP^2$ by a nonzero
closed rational $1$-form.
\end{theorem}

In the proof we shall need the following:
\begin{lemma}[Transverse M\"obius-group lemma]
\label{lem:mobius-separated-curves}
Let
\[
G\subset\operatorname{PGL}(2,\CC)
\]
be a group of M\"obius transformations. Assume that there exists an
infinite family
\[
\{\gamma_j\}_{j\in J}
\]
of pairwise disjoint, connected, compact real-analytic subsets of
$\PP^1$, each of pure real dimension one, such that
\[
g(\gamma_j)=\gamma_j
\]
for every $g\in G$ and every $j\in J$. Then there exists a nonzero
rational $1$-form $\theta$ on $\PP^1$ such that
\[
g^*\theta=\theta
\qquad\text{for every }g\in G.
\]
In particular, $\theta$ is closed.
\end{lemma}

\begin{proof}
Choose three distinct members
\[
\gamma_1,\gamma_2,\gamma_3
\]
and points $z_j\in\gamma_j$. Since the sets $\gamma_j$ are compact and
pairwise disjoint, their mutual spherical distances are positive.

Let $(g_\nu)$ be any sequence in $G$. Since
\[
g_\nu(z_j)\in\gamma_j,
\qquad j=1,2,3,
\]
after passing to a subsequence we may assume that
\[
g_\nu(z_j)\longrightarrow w_j\in\gamma_j,
\qquad j=1,2,3.
\]
The points $w_1,w_2,w_3$ are pairwise distinct. Since a M\"obius
transformation is uniquely determined by the images of three distinct
points, the sequence $(g_\nu)$ has a convergent subsequence in
$\operatorname{PGL}(2,\CC)$. Hence
\[
K=\overline G
\]
is a compact subgroup of $\operatorname{PGL}(2,\CC)$. Passing to the
limit in $g(\gamma_j)=\gamma_j$, we also have
\[
k(\gamma_j)=\gamma_j
\qquad\text{for every }k\in K\text{ and every }j.
\]

If $K$ is finite, let
\[
q:\PP^1\longrightarrow\PP^1/K
\]
be the quotient map. Since $\PP^1/K\simeq\PP^1$, choose a nonzero
rational $1$-form $\beta$ on the quotient and put
\[
\theta=q^*\beta.
\]
Then $\theta$ is a nonzero rational $K$-invariant $1$-form, hence it is
$G$-invariant.

Assume now that $K$ is infinite. Every compact subgroup of
$\operatorname{PGL}(2,\CC)$ is M\"obius-conjugate to a subgroup of
$\operatorname{PSU}(2)$. Since $K$ preserves the proper one-dimensional
sets $\gamma_j$, it cannot be all of $\operatorname{PSU}(2)$. Hence its
identity component $K^0$ is a circle subgroup. After a M\"obius change
of coordinate we may assume that
\[
K^0=\{z\mapsto e^{it}z;\ t\in\mathbb R\}.
\]

Each $\gamma_j$ is connected, $K^0$-invariant, and of pure real
dimension one. The nontrivial $K^0$-orbits are the circles
\(
|z|=r.
\)
If some $\gamma_j$ contained orbits with more than one value of $r$, it
would contain a real two-dimensional annular subset, contrary to its
pure real dimension one. Thus
\(
\gamma_j=\{|z|=r_j\}
\)
for some $r_j>0$.

It remains to show that $K$ itself is contained in this circle group.
An element of the normalizer of $K^0$ which is not in $K^0$ interchanges
the fixed points $0$ and $\infty$ and has the form
\[
z\longmapsto\frac{a}{z}.
\]
Such a map sends
\[
\{|z|=r\}
\quad\text{to}\quad
\{|z|=|a|/r\}.
\]
It can therefore preserve individually at most one circle of the
family, namely the circle satisfying $r^2=|a|$. Since the family
$\{\gamma_j\}$ is infinite and every member is preserved individually
by every element of $K$, no such element can belong to $K$. Therefore
\[
K\subset\{z\mapsto e^{it}z\}.
\]
Consequently
\[
\theta=\frac{dz}{z}
\]
is invariant under $K$, and hence under $G$. Pulling $\theta$ back by
the initial M\"obius conjugacy gives the required nonzero rational
$G$-invariant $1$-form on $\PP^1$. Since every meromorphic $1$-form on
a Riemann surface is closed, the proof is complete.
\end{proof}

\begin{proof}[Proof of Theorem~\ref{thm:separated-leviflats-closed-form}]
Let
\[
G\subset\operatorname{PGL}(2,\CC)
\]
be the projective monodromy group of the analytically continued minimal
transverse parameter. For each $j$, choose the connected component of the local transverse
trace specified in the continuation hypothesis and denote its continued
image by
\[
\gamma_j\subset\PP^1.
\]
By assumption, the sets $\gamma_j$ are connected, compact, pairwise
disjoint real-analytic subsets of pure real dimension one, and each of them
is preserved by every element of $G$.

By Lemma~\ref{lem:mobius-separated-curves}, there exists a nonzero
rational $1$-form $\theta$ on $\PP^1$ such that
\[
A^*\theta=\theta
\qquad\text{for every }A\in G.
\]
Choose local branches $F_\alpha$ of the analytically continued minimal
parameter and define
\(
\eta_\alpha=F_\alpha^*\theta.
\)
If
\(
F_\beta=A_{\beta\alpha}\circ F_\alpha,
\qquad A_{\beta\alpha}\in G,
\)
then
\(
\eta_\beta
 =F_\alpha^*(A_{\beta\alpha}^*\theta)
 =F_\alpha^*\theta
 =\eta_\alpha.
\)
Thus the local forms glue on $\PP^2\setminus\Sing(\F)$ to a
single meromorphic $1$-form $\eta$ defining $\F$ there. Since
$\Sing(\F)$ is finite, $\eta$ extends meromorphically across it.
On the projective surface $\PP^2$, every meromorphic $1$-form is
rational. Moreover $d\theta=0$ on the parameter sphere, hence
$d\eta=0$. Therefore $\F$ is defined by a nonzero closed rational
$1$-form.
\end{proof}

\begin{remark}[One relevant dicritical point]
\label{cor:separated-one-dicritical}
In Theorem~\ref{thm:separated-leviflats-closed-form}, suppose that $p$
is the only dicritical singularity involved in the separated family.
Then no compatibility condition between different dicritical points is
needed: it is enough that one minimal local parameter at $p$ continue
projectively throughout $\PP^2\setminus\Sing(\F)$ and that the associated
connected transverse parameter sets satisfy the separation and individual
invariance hypotheses of Theorem~\ref{thm:separated-leviflats-closed-form}.
The conclusion is again that $\F$ is defined by a nonzero closed rational
$1$-form.
\end{remark}

\section{Singular Levi-flats on Stein manifolds}
\label{sec:stein}

The semi-global argument is not specific to projective space. In this
section we formulate its Stein-manifold version and indicate the
additional Hartogs hypothesis under which the first integral becomes
global.

\begin{theorem}[Stein-manifold version]
\label{thm:stein}
Let $X$ be a Stein manifold of complex dimension $n\ge3$, let $\F$ be
a codimension-one holomorphic foliation on $X$, and let $M\subset X$
be a connected irreducible real-analytic variety of pure real codimension one 
invariant by $\F$. Assume:

\begin{enumerate}[label=\textup{(\roman*)}]
\item
\[
\Sing(M)\subset\Sing(\F);
\]

\item all dicritical singularities of $\F$ contained in $M$ belong to
an irreducible codimension-two analytic subset
\[
D\subset M\cap\Sing(\F),
\]
and $D$ is generically of local product type;

\item there exists a pencil singularity for $M$ in the local-product
locus of $D$;

\item the generic transverse parameter along $D^\circ$ has no
transverse degree jump along $D$ in the sense of
Definition~\ref{def:no-degree-jump}.
\end{enumerate}

Then there exist an open neighborhood $N\supset M$ and a nonconstant
meromorphic first integral
\[
G\in\Mcal(N)
\]
of $\F|_N$.

If, in addition, there exists a compact set $K\subset X$ such that
\[
X\setminus K\subset N
\qquad\text{and}\qquad
X\setminus K\ \text{is connected},
\]
then $G$ extends meromorphically to $X$, and $\F$ admits a nonconstant
global meromorphic first integral on $X$.
\end{theorem}

\begin{proof}[Proof of Theorem~\ref{thm:stein}]
The semi-global argument used in the proof of Theorem~\ref{thm:semiglobal} gives a
nonconstant meromorphic first integral
\[
G\in\Mcal(N)
\]
on a neighborhood $N\supset M$.

Assume now that there is a compact set $K\subset X$ such that
\[
X\setminus K\subset N
\qquad\text{and}\qquad
X\setminus K\ \text{is connected}.
\]
The restriction of $G$ to $X\setminus K$ extends meromorphically
through $K$ by the meromorphic Hartogs extension theorem on the Stein
manifold $X$. We obtain
\[
\widetilde G\in\Mcal(X).
\]
By uniqueness of meromorphic continuation, $\widetilde G$ agrees with
$G$ wherever both are defined. Since
\[
dG\wedge\omega=0
\]
near $M$, analytic continuation gives
\[
d\widetilde G\wedge\omega=0
\]
wherever $\omega$ defines $\F$. Thus $\widetilde G$ is a nonconstant
global meromorphic first integral of $\F$.
\end{proof}

\section{The case of several dicritical singularities}
\label{sec:compatible-dicritical}

We conclude with variants in which several dicritical singularities carry
compatible local meromorphic parameters. The same compatibility mechanism
produces a semi-global first integral and yields the corresponding
projective consequences, including the proof of the plane characterization.

\subsection{The higher-dimensional case}
\label{subsec:higher-dim-dicritical}

\begin{corollary}[All singularities on $M$ dicritical]
\label{cor:higher-dimensional-dicritical}
Let $\F$ be a codimension-one holomorphic foliation on $\PP^n$,
$n\ge3$, and let
\[
M\subset\PP^n
\]
be an irreducible real-analytic variety of pure real codimension one  invariant by
$\F$. Assume that
\[
\Sing(M)\subset\Sing(\F)
\]
and that every singularity of $\F$ contained in $M$ is dicritical.

Choose a finite collection of neighborhoods
\(
U_1,\ldots,U_r
\)
covering $\Sing(\F)\cap M$, and, after shrinking if necessary, choose
on each $U_j$ a minimal meromorphic first integral
\(
H_j:U_j\dashrightarrow\PP^1
\)
for $\F|_{U_j}$.

Assume that whenever a generic global Levi leaf
$L\subset M_{\reg}$ enters a sufficiently small neighborhood of
$\Sing(\F)\cap M$ contained in $U_j$, all the corresponding local
branches of $L$ determine the same value
\[
\lambda_j(L)\in\PP^1
\]
of $H_j$. Assume moreover that the local parameters are projectively
compatible: there exist transformations
\[
A_{ji}\in\operatorname{PGL}(2,\CC)
\]
depending only on the pair $(i,j)$, and not on the Levi leaf, such that
whenever a generic global Levi leaf $L$ is represented in both
$U_i$ and $U_j$ one has
\[
\lambda_j(L)=A_{ji}\bigl(\lambda_i(L)\bigr),
\]
and the maps $A_{ji}$ satisfy
\[
A_{ii}=\operatorname{id},
\qquad
A_{kj}\circ A_{ji}=A_{ki}
\]
whenever all terms are defined.

Then $\F$ admits a nonconstant rational first integral on $\PP^n$.
\end{corollary}

\begin{proof}[Proof of Corollary~\ref{cor:higher-dimensional-dicritical}]
Let $L\subset M_{\reg}$ be a Levi leaf. Its closure
\(
\overline L\subset M\subset\PP^n
\)
is a nonempty compact $\F$-invariant set. By a theorem of
Lins Neto on nontrivial minimal sets in projective space
\cite{LinsNeto1999}, a codimension-one holomorphic foliation on
$\PP^n$, $n\ge3$, has no nontrivial minimal set. Consequently
$\overline L$ meets $\Sing(\F)$. Since every singularity of $\F$
contained in $M$ is dicritical, every Levi leaf accumulates on the
dicritical locus in $M$.

The finite family $U_1,\ldots,U_r$ therefore meets the accumulation
set of every Levi leaf. The constancy assumption allows the local
parameters $H_j$ to be transported along the Levi foliation, and the
M\"obius compatibility and cocycle relations identify the resulting
parameters on overlaps. Thus the same holonomy argument as in
Proposition~\ref{prop:nearMreg} produces a single-valued meromorphic
first integral on a neighborhood of $M_{\reg}$.

Each $H_j$ supplies the extension through the singular set covered by
$U_1,\ldots,U_r$. Hence the first integral is meromorphic on a
neighborhood $N$ of all of $M$. For each connected component $\Omega$
of $\PP^n\setminus M$, Proposition~\ref{prop:pseudoconvex} and
Takeuchi's theorem imply that $\Omega$ is Stein. Applying
Lemma~\ref{lem:boundary-hartogs} to $G|_{N\cap\Omega}$ and shrinking $N$
near $M$ if necessary extends the first integral through every
complementary component. The resulting global meromorphic function on
$\PP^n$ is necessarily rational.
\end{proof}

\subsection{A finite dicritical set in the projective plane}
\label{subsec:finite-dicritical-plane}

For the plane we use the same local information. Namely, each dicritical
singularity $p_j$ of $\F$ contained in an invariant real-analytic
hypersurface admits, by Cerveau--Lins Neto and Brunella, a nonconstant
meromorphic first integral for the germ $(\F,p_j)$. For each
$j=1,\ldots,r$, choose a local meromorphic first integral
\(
H_j:(U_j,p_j)\dashrightarrow\PP^1
\)
of the germ $(\F,p_j)$.

\begin{corollary}[Finite dicritical set in $\PP^2$]
\label{cor:finite-dicritical-plane}
Let $\F$ be a codimension-one holomorphic foliation on $\PP^2$, and
let
\(
M\subset\PP^2
\)
be an irreducible real-analytic variety of pure real codimension one  invariant by $\F$. Assume
that
\(
\Sing(M)\subset\Sing(\F).
\)
Let $p_1,\ldots,p_r\in M\cap\Sing(\F)$ be all the dicritical singularities
and, for each $j$, let
\[
H_j:(U_j,p_j)\dashrightarrow\PP^1
\]
be a local meromorphic first integral of the germ
$(\F,p_j)$. Assume that every Levi leaf $L\subset M_{\reg}$
accumulates at least one of the points $p_j$.

Assume that, for every $j$ and every generic global Levi leaf $L$
approaching $p_j$, the restriction $H_j|_{L\cap U_j}$ is constant;
denote its value by
\[
\lambda_j(L)\in\PP^1.
\]
Assume moreover that there exist transformations
\[
A_{ji}\in\operatorname{PGL}(2,\CC)
\]
depending only on the pair $(i,j)$, and not on the Levi leaf, such that
whenever a generic global Levi leaf $L$ accumulates at both $p_i$
and $p_j$ one has
\[
\lambda_j(L)=A_{ji}\bigl(\lambda_i(L)\bigr),
\]
and that the transformations satisfy
\[
A_{ii}=\operatorname{id},
\qquad
A_{kj}\circ A_{ji}=A_{ki}
\]
whenever all terms are defined.

Then $\F$ admits a nonconstant rational first integral on $\PP^2$.
\end{corollary}

\begin{proof}[Proof of Corollary~\ref{cor:finite-dicritical-plane}]
By hypothesis every Levi leaf accumulates at least one of the
dicritical points $p_1,\ldots,p_r$. Starting from a local first
integral $H_j$ at any such point reached by a given leaf, the
constancy condition makes the holonomy continuation single-valued
near that leaf. The projective compatibility maps $A_{ji}$ and their
cocycle relations identify these local continuations whenever their
domains overlap. They therefore glue to a meromorphic first integral
on a neighborhood of $M_{\reg}$.

At each $p_j$ the corresponding normalized local first integral
extends the glued function meromorphically through $p_j$. Any
remaining singular point of $\F$ on $M$ is non-dicritical and is
treated by Lemma~\ref{lem:nondicritical}. Thus the first integral is
meromorphic on a neighborhood of all of $M$.

The local pseudoconvexity argument of Section~\ref{sec:pseudoconvexity} and the meromorphic
extension argument on the complementary Stein domains globalize this
function to $\PP^2$. The resulting global meromorphic function is
rational.
\end{proof}

\subsection{The case of a single dicritical singularity}
\label{subsec:single-dicritical-plane}

\begin{proof}[Proof of Corollary~\ref{cor:single-dicritical-plane}]
This is the special case of Corollary~\ref{cor:finite-dicritical-plane} with $r=1$.
Equivalently, Proposition~\ref{prop:nearMreg} gives the first integral
near $M_{\reg}$, the local first integral at $p$ supplies the
extension at the only dicritical singularity, and
Lemma~\ref{lem:nondicritical} treats the remaining singular points.
The projective extension argument then gives a rational first
integral.
\end{proof}

\begin{example}
The elementary rational pencil
\(
R=\frac{z_0}{z_1}
\)
on $\PP^2$ preserves
\(
M=\overline{R^{-1}(\mathbb RP^1)}.
\)
The unique dicritical point is
\(
p=[0:0:1],
\)
and the Levi leaves are
\[
z_0-\lambda z_1=0,\qquad \lambda\in\mathbb RP^1.
\]
Thus Corollary~\ref{cor:single-dicritical-plane} applies.
\end{example}

\subsection{Characterization of primitive rational pencils on \texorpdfstring{$\PP^2$}{P2}}
\label{subsec:plane-characterization}

\begin{proof}[Proof of Theorem~\ref{thm:plane-characterization}]
Assume first that $\F$ admits a primitive rational first integral
\(
R:\PP^2\dashrightarrow\PP^1.
\)
Choose a smooth real-analytic Jordan curve
\(
\gamma\subset\PP^1
\)
of regular values of the resolved rational map and set
\(
M=\overline{R^{-1}(\gamma)}.
\)
The generic fiber is connected and, over regular values of the resolved
map, smooth; hence the corresponding Levi leaves are the connected
fibers over $\gamma$. The hypersurface $M$ is invariant by $\F$ and is
smooth away from the base locus of the pencil. Let
$p_1,\ldots,p_r$ be all dicritical base points of $R$ belonging to $M$.
For each $j$, take
\(
H_j=R|_{U_j}:(U_j,p_j)\dashrightarrow\PP^1.
\)
These local meromorphic first integrals need not be minimal; this is why
minimality is not imposed in condition \textup{(ii)}. On every generic
global Levi leaf $L$, each $H_j$ is constant wherever $L$ approaches
$p_j$, with value equal to the global value $R(L)$. Thus we may take
\(
A_{ji}=\operatorname{id}
\)
(after using the same global coordinate on the target $\PP^1$ for all
$j$), and all compatibility conditions in \textup{(ii)} follow.

Conversely, assume \textup{(ii)}. The local constancy and projective
compatibility hypotheses are precisely the hypotheses required by
Corollary~\ref{cor:finite-dicritical-plane}; its proof uses only the
chosen local meromorphic first integrals and does not require their
minimality. Hence $\F$ admits a nonconstant rational first integral
on $\PP^2$. Replacing it by the primitive generator furnished by the
Stein factorization of the associated rational map gives a primitive
rational first integral.
\end{proof}

\section*{Declarations}

\noindent\textbf{Funding.} The author has been partially supported by Funda\c{c}\~ao Get\'ulio Vargas through its School of Applied Mathematics (EMAp).

\noindent\textbf{Competing interests.} The author declares no competing interests.

\noindent\textbf{Declaration of AI use.} The author used ChatGPT (OpenAI) to assist with tasks as: language editing, organization, and checking the clarity of the manuscript.  The author reviewed and verified the final text, including all mathematical statements, proofs, and references, and takes full responsibility for its content.


\begin{thebibliography}{99}

\bibitem{BarrettForness}
D.~E.~Barrett and J.~E.~Forn{\ae}ss,
\emph{On the smoothness of Levi-foliations},
Publ. Mat. \textbf{32} (1988), no.~2, 171--177.


\bibitem{Brunella}
M.~Brunella,
\emph{Some remarks on meromorphic first integrals},
Enseign. Math. \textbf{58} (2012), no.~3--4, 315--324.




\bibitem{CAS}
O.~Calvo-Andrade and M.~G.~Soares,
\emph{Chern numbers of a Kupka component},
Ann. Inst. Fourier (Grenoble) \textbf{44} (1994), no.~4, 1219--1236.


\bibitem{CamachoLinsNeto}
C.~Camacho and A.~Lins Neto,
\emph{Geometric Theory of Foliations},
Birkh\"auser Boston, Boston, MA, 1985.



\bibitem{CM}
D.~Cerveau and J.-F.~Mattei,
\emph{Formes int\'egrables holomorphes singuli\`eres},
Ast\'erisque \textbf{97} (1982).

\bibitem{CamachoLinsNetoSad}
C.~Camacho, A.~Lins Neto and P.~Sad,
\emph{Foliations with algebraic limit sets},
Ann. of Math. (2) \textbf{136} (1992), no.~2, 429--446.

\bibitem{CLN}
D.~Cerveau and A.~Lins Neto,
\emph{Local Levi-flat hypersurfaces invariants by a codimension one
holomorphic foliation},
Amer. J. Math. \textbf{133} (2011), no.~3, 677--716.

\bibitem{CLNKupka}
D.~Cerveau and A.~Lins Neto,
\emph{Codimension one foliations in $\mathbb{CP}^n$, $n\ge3$, with
Kupka components},
Ast\'erisque \textbf{222} (1994), 93--133.

\bibitem{DummitFoote}
D.~S.~Dummit and R.~M.~Foote,
\emph{Abstract Algebra}, 3rd ed.,
John Wiley \& Sons, 2004.

\bibitem{LinsNeto1999}
A.~Lins Neto,
\emph{A note on projective Levi flats and minimal sets of algebraic
foliations},
Ann. Inst. Fourier (Grenoble) \textbf{49} (1999), no.~4, 1369--1385.

\bibitem{LNproduct}
A.~Lins Neto,
\emph{Local transversely product singularities},
Ann. Henri Lebesgue \textbf{4} (2021), 485--503.

\bibitem{FMRdyn}
A.~Fern\'andez-P\'erez, R.~Mol and R.~Rosas,
\emph{On the dynamics of foliations in $\mathbb P^n$ tangent to
Levi-flat varieties},
preprint, arXiv:1403.4802.

\bibitem{FMRsing}
A.~Fern\'andez-P\'erez, R.~Mol and R.~Rosas,
\emph{On singular real analytic Levi-flat foliations},
Asian J. Math. \textbf{24} (2020), no.~6, 1007--1028.

\bibitem{FMRchow}
A.~Fern\'andez-P\'erez, R.~Mol and R.~Rosas,
\emph{Chow's theorem for real analytic Levi-flat hypersurfaces},
Bull. Sci. Math. \textbf{179} (2022), 103169.

\bibitem{GalindoMonserrat}
C.~Galindo and F.~Monserrat,
\emph{The Poincar\'e problem, algebraic integrability and dicritical divisors},
J. Differential Equations \textbf{256} (2014), no.~11, 3614--3633.

\bibitem{Jouanolou}
J.-P.~Jouanolou,
\emph{Equations de Pfaff alg\'ebriques},
Lecture Notes in Mathematics, vol.~708,
Springer-Verlag, Berlin, 1979.

\bibitem{HM}
H.~Hironaka and H.~Matsumura,
\emph{Formal functions and formal embeddings},
J. Math. Soc. Japan \textbf{20} (1968), no.~1--2, 52--82.

\bibitem{MR}
J.~Martinet and J.-P.~Ramis,
\emph{Classification analytique des equations differentielles non
lineaires resonnantes du premier ordre},
Ann. Sci. Ecole Norm. Sup. (4) \textbf{16} (1983), no.~4, 571--621.

\bibitem{MerkerPorten}
J.~Merker and E.~Porten,
\emph{The Hartogs extension theorem on $(n-1)$-complete complex spaces},
J. Reine Angew. Math. \textbf{637} (2009), 23--39.

\bibitem{Paul1999}
E.~Paul,
\emph{Feuilletages holomorphes singuliers \`a holonomie r\'esoluble},
J. Reine Angew. Math. \textbf{514} (1999), 9--70.

\bibitem{Sad1995}
P.~Sad,
\emph{Dicritical singularities in foliations with algebraic limit sets},
Bol. Soc. Brasil. Mat. (N.S.) \textbf{26} (1995), no.~1, 47--56.

\bibitem{GR}
H.~Grauert and R.~Remmert,
\emph{Coherent Analytic Sheaves},
Grundlehren der mathematischen Wissenschaften, vol.~265,
Springer-Verlag, Berlin, 1984.



\bibitem{SerreDuality}
J.-P.~Serre,
\emph{Un th\'eor\`eme de dualit\'e},
Comment. Math. Helv. \textbf{29} (1955), 9--26.

\bibitem{Takeuchi}
A.~Takeuchi,
\emph{Domaines pseudoconvexes infinis et la m\'etrique riemannienne
dans un espace projectif},
J. Math. Soc. Japan \textbf{16} (1964), 159--181.


\bibitem{CamachoScarduaLiouvillian}
C.~Camacho and B.~Azevedo Sc\'ardua,
\emph{Holomorphic foliations with Liouvillian first integrals},
Ergodic Theory Dynam. Systems \textbf{21} (2001), no.~3, 717--756.

\bibitem{ScarduaBoundedGeometry}
B.~C.~A.~Sc\'ardua,
\emph{Complex vector fields having orbits with bounded geometry},
Tohoku Math. J. (2) \textbf{54} (2002), no.~3, 367--392.

\bibitem{CerveauLinsNetoStructural}
D.~Cerveau and A.~Lins Neto,
\emph{A structural theorem for codimension-one foliations on
$\PP^n$, $n\ge3$, with an application to degree-three foliations},
Ann. Sc. Norm. Super. Pisa Cl. Sci. (5) \textbf{12} (2013), no.~1, 1--41.

\bibitem{LeblSingularSet}
J.~Lebl,
\emph{Singular set of a Levi-flat hypersurface is Levi-flat},
Math. Ann. \textbf{355} (2013), no.~3, 1177--1199.

\end{thebibliography}
\end{document}